\documentclass[10pt, a4paper]{article}
\usepackage{lipsum}
\usepackage{amsfonts, amsthm, amsmath, amssymb}
\usepackage{graphicx}
\usepackage{psfrag}

\usepackage{epstopdf}
\usepackage{algorithmic}
\usepackage{authblk}
\usepackage{url}
\usepackage[hidelinks]{hyperref}
\usepackage[noabbrev]{cleveref}
\ifpdf
\DeclareGraphicsExtensions{.eps,.pdf,.png,.jpg}
\else
\DeclareGraphicsExtensions{.eps}
\fi
\usepackage{graphicx,epstopdf} % <- Preamble
\usepackage[caption=false]{subfig} % <- Preamble
\usepackage{pgfplots}
\usepackage{tikz}
\usetikzlibrary{matrix}

\newcommand{\E}{\varphi}
\newcommand{\oX}{\overline{X}}
\newcommand{\oa}{\overline{a}}
\newcommand{\ob}{\overline{b}}
\newcommand{\oc}{\overline{c}}
\newcommand{\oDelta}{\overline{\Delta}}
\newcommand{\N}{\ensuremath{\mathbb{N}}}
\newcommand{\R}{\ensuremath{\mathbb{R}}}
\newcommand{\C}{\ensuremath{\mathbb{C}}}

\theoremstyle{definition}
\newtheorem{theorem}{Theorem}[section]

\newtheorem{lemma}{Lemma}[section]

\theoremstyle{definition}
\newtheorem{definition}{Definition}[section]

\theoremstyle{remark}
\newtheorem{remark}{Remark}

\title{Numerical Solution of Free Stochastic Differential Equations}

\providecommand{\keywords}[1]
{
	\small	
	\noindent\textbf{\textbf{Keywords }} #1
}

\providecommand{\amscodes}[1]
{
	\small	
	\noindent\textbf{\textbf{AMS Codes }} #1
}

\date{\today} % Comment this line to show today's date

\author[1, 2]{Georg Schluechtermann}
\author[2]{Michael Wibmer}

\affil[1]{\footnotesize Faculty of Mathematics, Informatics and Statistics, LMU Munich, Germany}
\affil[2]{\footnotesize Faculty of Mechanical, Aeronautical and Automotive Engineering,\\ University of Applied Sciences, Munich, Germany}

\begin{document}

\maketitle

% REQUIRED
\begin{abstract}
This paper derives a free analog of the Euler-Maruyama method (fEMM) to numerically approximate solutions of free stochastic differential equations (fSDEs). Simply speaking fSDEs are stochastic differential equations in the context of non-commutative random variables (e.g. large random matrices). By applying the theory of multiple operator integrals we derive a free It\^{o} formula from Taylor expansion of operator valued functions. Iterating the free It\^{o} formula allows to motivate and define fEMM. Then we consider weak and strong convergence in the fSDE setting and prove strong convergence order of $\frac{1}{2}$ and weak convergence order of ${1}$. Numerical examples support the theoretical results and show solutions for equations where no analytical solution is known.
\end{abstract}

~\\~
\amscodes{46L53, 46L54, 60H10, 65C30}
\section{Introduction}
Nowadays random matrices appear in a broad range of applications (e.g. \cite{Bouchaud2015}, \cite{freeCIR}, \cite{ADHIKARI2022108260}, \cite{XIAO2017941}, \cite{anderson_guionnet_zeitouni_2009}, \cite{MingoSpeicher2017},  \cite{Maecki2019UniversalityCF}, \cite{Johnstone8412585}, \cite{Zhang2015}, \cite{stone2018}, \cite{Soize2017}). Random matrices with certain spectral properties can be obtained as solutions of free stochastic differential equations (fSDEs). A 
fSDE is an equation of the form
\begin{equation}
	dX_t = a(X_t)dt + b(X_t)dW_tc(X_t),
\end{equation}
where the unknown $X_t$ is an operator valued process $(X_t)_{t\geq 0}$, $(W_t)_{t\geq0}$ is a free Brownian motion and $a,b,c$ are appropriate operator valued functions. At a first sight, one may think about $X_t$ as random matrices of large dimension. Free Brownian motion can be viewed of as the large $N$ limit of Brownian motions on $N\times N$ hermitian matrices (\cite{Biane1998}). By taking the large $N$ limit fSDEs are formulated in some appropriate von Neumann algebra $\mathcal{A}$. Free probability theory and free stochastic processes set up the background for a solution theory of fSDEs (\cite{kargin}). The notion of freeness is the carryover of the notion of independence of random variables to the non-commutative context. For an introduction to free probability and free stochastic processes we refer to  \cite{speicher2001free}, \cite{voicudykemanica}, \cite{MingoSpeicher2017}, \cite{anderson_guionnet_zeitouni_2009}, \cite{TaoIntroRMT}, \cite{stammvoicuweber}. A short introduction will be given at the beginning of this paper. Free stochastic calculus first appeared in \cite{Speicher1990} and was further developed by \cite{kummererspeicher}, \cite{BianezbMATH01003147} and \cite{Biane1998}. The notion of free stochastic processes, free Brownian motion and a free analog of the It\^{o} formula were introduced in \cite{Biane1998}. For definition of the free It\^{o} integral, we refer to \cite{anshelevic}. Free stochastic processes form an active research area, we refer to \cite{barnodorff10.2307/3318705}, \cite{speicher2001free}, \cite{Biane1998-2}, \cite{AnGao}, \cite{GAO2006177}, \cite{ZHAOZHI}, \cite{freeCIR}.
fSDEs first appeared in \cite{kummererspeicher}, \cite{BianezbMATH01003147} and \cite{BIANESPEICHER2001581}, where such equations are motivated by studying large $N$ quantum field theory and corresponding matrix models. A first existence theory and a variety of fSDEs were analytically studied in \cite{kargin}. Recently in \cite{freeCIR} a free variant of the Cox-Ingersoll-Ross model (\cite{HighamKloeden2020}) is considered in the context of financial mathematics.\\
To the best of our knowledge the numerical solution of free stochastic differential equations has not yet been studied before. In \cite{BIANESPEICHER2001581} an Euler-like method was applied to prove the existence of a solution of a special fSDE and furthermore regularity results were obtained in the operator norm.\\
 The purpose of this paper is to develop, analyze and apply a method for numerical approximation of fSDEs. As in the classical case we start by developing a free analog of the Euler-Maruyama method (fEMM). The derivation of the method will be stated and carried out by considering the fSDE in a von Neumann Algebra $\mathcal{A}$ with faithful unital normal trace. The free It\^{o} formula \cite[Theorem 4.1.2]{Biane1998} will play a central role in this context. The free analog to the classical It\^{o} formula was derived by applying the concept of double operator integrals on $\mathcal{B}(\mathcal{H})$. The perturbation theory of operator valued functions and extension to multiple operator integrals was further developed, which gives rise to new conceptual and technical tools. We make use of these developments (\cite{Skripka2019}, \cite{pisier_2003}) to reformulate the free It\^{o} formula. The formalism defined in \cite{azamov_carey_dodds_sukochev_2009} allows for a consistent and effective formulation of the free It\^{o} formula. Based on a Taylor polynomial of operator functions (\cite{azamov_carey_dodds_sukochev_2009}, \cite{Skripka2019}) we derive the It\^{o} formula (\cite[Proposition 4.3.4]{Biane1998}) directly from the stochastic product rules. The proof of this proposition relies on approximating via polynomials. Alternatively we derive the free It\^{o} formula via Taylor expansions of the corresponding operator valued functions. Then we are in the position to formulate an iterated version of the free It\^{o} formula, which allows the motivation and definition of a free analog of the Euler-Maruyama method (fEMM). We will give proofs for strong and weak convergence properties. It turns out that strong convergence is of order $\frac{1}{2}$. This is mainly driven by the fact that the $L_2(\varphi)$-norm of the stochastic integral $\int_{\Delta t} a_tdW_tb_t$ is of $O(\sqrt{\Delta t})$. Additionally, the coefficient functions $a,b,c$ of the fSDE need to be operator Lipschitz in $L_2(\varphi)$. The existence theory in \cite{kargin} requires $a,b,c$ to be Lipschitz in operator norm of $\mathcal{A}$. Weak convergence of order one is proven under the assumptions $a,b,c$ are uniformly bounded in $\mathcal{A}$ and belonging to certain proper spaces $W_n(\R)$ (see \cite{azamov_carey_dodds_sukochev_2009},\cite{Skripka2019}). Several examples show the capability of fEMM to approximate solutions of fSDEs. We will verify that fEMM can well numerically approximate spectral properties of the solution of the underlying fSDE.\\
To be able to implement fEMM as a numerical method on a computer, it is necessary to consider the fEMM in a von Neumann algebra of random matrices. We will show that the large $N$ limit of fEMM on matrix level leads to a fEMM defined in a finite von Neumann algebra (just as limits of random matrices end up in a infinite dimensional von Neumann Algebra). Weak and strong convergence properties do hold in any appropriate von Neumann algebra. Since random matrices form itself a von Neumann algebra, the convergence properties of fEMM in $\mathcal{A}$ carry over to the numerical algorithm. Additionally we show that both limits $N\rightarrow\infty$ and $\Delta t\rightarrow 0$ commute. We give examples by applying fEMM to equations considered in \cite{kargin} and show, that both limits of step size and matrix size commute and converge in distribution to the distribution of the solution. We will numerically verify the theoretically obtained properties of weak and strong convergence.\\
The paper is organized as follows. The necessary ingredients to define fSDEs are summarized in \cref{sec:FreeStochCalc} and \cref{sec:fSDE}. In \cref{sec:freeItoFuncForm} we formulate the free It\^{o} formula in the context of the framework of multiple operator integrals developed in \cite{azamov_carey_dodds_sukochev_2009}. We then derive an iterated It\^{o} formula which allows to motivate the free Euler-Maruyama method. Next, the new numerical algorithm is defined in \cref{sec:fEMM}. The main results regarding convergence properties are given in \cref{sec:main}. Section \ref{sec:Examples} shows examples of the numerical approximation for fSDEs.
%  The notion of freeness origins in the work of Voiculescu \cite{voiculescu2}, \cite{voicudykemanica} which is the non-commutative counterpart of the notion of independence in classical probability. 
\section{Free Stochastic Calculus}
\label{sec:FreeStochCalc}

In this chapter we summarize the main results of free probability theory and free stochastic calculus. Free stochastic calculus was initiated by \cite{Speicher1990} and developed in a series of papers in \cite{kummererspeicher}, \cite{Biane1998}, \cite{anshelevic} and \cite{BianezbMATH01003147}. We will introduce the notion of free Brownian motion, free stochastic calculus which includes a free analog of the classical It\^{o} formula.

\subsection{Free Probability Theory}
\label{sec:FreeProbaTheo}

Consider a classical probability space $(\Omega, \mathcal{F}, \mu)$ and random variables as measurable functions $X:\Omega \rightarrow \R$. By taking an algebraic viewpoint these random variables $X$ form an (commutative) algebra, where it is possible to assign expectations $\mathbb{E}(X)$ to each random variable. This change of viewpoint allows to consider cases, where the random variables are non-commutative. The space $\mathcal{M}_N(\R)=L^\infty\left(\Omega,\mu,\text{Mat}_N(\R)\right)$ builds up a star-algebra with the unit matrix as identity and $\varphi(M)=\frac{1}{N}\mathbb{E}(\text{tr}(M))$ as a trace. By help of non-commutative algebras it is possible to develop non-commutative probability theory. The limits $N\rightarrow\infty$ can be handled properly in algebraic structures and leads to fruitful concepts. It turns out that non-commutative probability theory is realized by using operator algebras such as von Neumann algebras or algebras of bounded operators on a Hilbert space. We refer to \cite{TaoIntroRMT}, \cite{Biane1998}, \cite{anderson_guionnet_zeitouni_2009} for setting up non-commutative probability theory and relations to random matrices. To be complete, we give the following general definition (see e.g. \cite{werner}). 
\begin{definition}
	A non-commutative probability space is a pair $(\mathcal{A}, \varphi)$, where $\mathcal{A}$ denotes a von Neumann operator algebra and $\varphi:\mathcal{A}\rightarrow \C$ a faithful unital normal trace.
\end{definition}
Since the trace is finite we may consider $\mathcal{A}$ as a subset of the predual $L_1(\varphi)$ of the von Neumann algebra $\mathcal{A}=L_\infty(\varphi)$. For $1\leq p<\infty$ we define $\|X\|_p=\varphi(|X|^p)^{\frac{1}{p}}$. By $\| \cdot \|$ we denote the usual operator norm in $\mathcal{A}$. Although the definition of a non-commutative probability space is rather abstract, once the concepts are stated, they turn into background when working on numerical methods. The notion of independence of classical random variables is extended to the non-commutative setting by the concept of freeness of subalgebras of $\mathcal{A}$. Let $\mathcal{A}_1,\dots \mathcal{A}_n$ be a family of $n\in\N$ subalgebras of $\mathcal{A}$. They are called freely independent (or simply free) in the sense of Voiculescu, if $\varphi\left(X_1X_2\dots X_m\right)=0$ 
whenever the following conditions
\begin{enumerate}
	\item $X_j\in\mathcal{A}_{i(j)} $, where $i(1)\neq i(2), i(2)\neq i(3), \dots , i(n-1)\neq i(n)$, $j=1,\dots,m$
	\item $\varphi(X_i)=0$ for all $i=1,\dots,n$
\end{enumerate} 
hold. If $X\in\mathcal{A}$ is a self-adjoint element, then there is a spectral measure $\mu$ on $\R$ so that the moments of $X$ are the same as the moments of the probability measure $\mu$ defined by
$
	\varphi(X^k)=\int_\R x^k d\mu(x).
$
An important role in the subsequent plays the Cauchy transform $G_X$ of $\mu$ defined by
$
	G_X(z)=\int_{\R} \frac{d\mu(x)}{x-z},
$
which is an analytic function defined on $\C^+$ with values in $\C^+$. The Cauchy transform $G_X$ is the expectation of the resolvent of $X$, i.e.
$
	G_X(z)=\varphi\left(\left(X-z\right)^{-1}\right).
$
The Cauchy transform carries all the properties of the spectral probability distribution of the self-adjoint operator $X$. In \cite{Biane1998} and \cite{kargin} it is shown how fSDEs can be handled by a corresponding deterministic partial differential equations of the Cauchy transform $G_X$. We will strongly depend on these results since it allows us to check the numerical results obtained in \cref{sec:fEMM} by the free stochastic Euler method defined.
\subsection{Free Brownian Motion}
\label{sec:FreeBrownMotion}
Motivated from the concept of classical Brownian motion the definition within non-commutative probability is as follows. Consider a von Neumann Algebra $\mathcal{A}$ with a faithful normal tracial state $\varphi:\mathcal{A}\rightarrow \C$.
A filtration $\mathbb{F}=(\mathcal{A}_t)_{t \geq 0}$ is a family of subalgebras $\mathcal{A}_t$ of $\mathcal{A}$ with $\mathcal{A}_s \subset \mathcal{A}_t$ for $s\leq t$. A free stochastic process is a family of elements $(X_t)_{t\geq 0}$ for which the increments $X_t - X_s$ are free with respect to the subalgebra $\mathcal{A}_s$. A process $(X_t)_{t\geq0}$ is called {\it adapted} to the filtration $\mathbb{F}$ if $X_t\in\mathcal{A}_t$ for all $t\geq 0$. 
\begin{definition}
	A free Brownian motion is a family of self-adjoint elements $(W_t)_{t\geq 0}$ which admit the properties
	\begin{enumerate}
		\item $W_0=0$.
		\item The increments $W_t-W_s$ are free from $\mathcal{W}_s$  for all $0\leq s<t$. The subalgebra $\mathcal{W}_s$ is generated by all $W_\tau$ with $\tau\leq s$.
		\item The increment $W_t-W_s$  is semicircle with mean $0$ and variance $t-s$ for all $0\leq s<t$.
	\end{enumerate}
\end{definition}
We define the filtration $\mathbb{F}=(\mathcal{W}_t)_{t\geq0}$ where $\mathcal{W}_t$ ist generated by all elements $W_s,s<t$.
\begin{remark}
	Free Brownian motion $(W_t)_{t\geq 0}$ can be viewed as large $N$ limit of $N\times N$ hermitian random matrices having classical independent Brownian motion entries $b_{ij}(t)$ (\cite{anshelevic}, \cite{Biane1998}). Considering the symmetric $N$-dimensional quadratic random matrix 
	$
		W^N_t := \frac{1}{\sqrt{N}}\left(b_{ij}(t)\right)_{N\times N},
	$
	the limit $\lim\limits_{N\rightarrow \infty}W_t^N$ defines an element $W_t$ in a von Neumann algebra $\mathcal{A}$ with trace
	$
		\varphi(\cdot) = \lim\limits_{N\rightarrow \infty}\mathbb{E}(\frac{1}{N}\text{tr}(\cdot)).
	$
\end{remark}

\subsection{Stochastic Integration with Respect to Free Brownian Motion}
\label{sec:FreeStochInt}

Let $(W_t)_{t\geq0}$ be a free Brownian motion. Let $a,b$ be mappings $[0,T]\rightarrow\mathcal{A}$ such that $\|a(t)\|\|b(t)\|\in L_2([0,T])$ and $a(t), b(t)\in\mathcal{W}_t$. We shorten the notation $a(t)=a_t,b(t)=b_t$ in the following, if there is no danger of confusion. Under these assumptions it is possible to define an  It\^{o}-style free stochastic integration with respect to free Brownian motion (\cite{Biane1998}, \cite{anshelevic}), written as 
\begin{equation}
	\int_0^t a_sdW_s b_s.
\end{equation}
For details of the definition and conditions for the existence and properties we refer to \cite{kargin}, \cite{anshelevic}, \cite{Biane1998}. The free stochastic integral fulfills a free analog of Burkholder-Gundy martingale inequalities (Section 3.2. in \cite{Biane1998}), i.e.
\begin{equation}\label{ineq:BurkholdGundy}
	\left\| \int_0^t a_sdW_sb_s\right\| \leq 2\sqrt{2}\left( \int_0^t \|a_s\|^2 \|b_s\|^2ds \right)^{\frac{1}{2}}.
\end{equation}
\subsection{Free It\^{o} Formula and - Process}
\label{sec:FreeIto}

An important ingredient in the development of numerical methods for fSDEs and their convergence properties is a free analog of the It\^{o}-formula (\cite[Section 4]{Biane1998}, \cite{kummererspeicher}, \cite{anshelevic}, \cite{kargin}). In terms of stochastic integrals the stochastic product rule is given in \cite[Theorem 4.1.2]{Biane1998} and can simply be written in differential form as (\cite{kargin})
\begin{equation}\label{freeItoFormulaAsInKargin}
	a_tdW_tb_t \cdot c_tdW_td_t= \varphi\left(b_tc_t\right)a_td_tdt.
\end{equation}
In the important case $a=c=d=1$ this yields the formal rules 
$dW_tb(X_t)dW_t = \varphi\left(b(X_t)\right)dt$ and $dW_tdW_t=dt.$
In the following we restrict ourselves to self-adjoint elements $a_t,b_t,c_t,d_t\in\mathcal{A}$ and denote the set of self-adjoint elements of $\mathcal{A}$ by $\mathcal{A}^{sa}$.
\begin{definition}\label{def:FreeItoProcess}
	Let $(W_t)_{t\geq0}$ be a free Brownian motion and $\mathbb{F}$ it's natural filtration. 
	An adapted mapping $X_t:[0,T]\rightarrow \mathcal{A}^{sa}$ is called a free It\^{o}-process, if there are operator valued functions $a_i,b_i,c_i:[0,T]\rightarrow\mathcal{A}^{sa}$ and an element $X_0\in \mathcal{A}^{sa}_0$ so that 
	\begin{equation}\label{intro-def-freeIto}
		X_t=X_0 + \int_0^t a(s)ds + \sum\limits_{i=0}^k\int_0^tb^i(s)dW_sc^i(s).
	\end{equation}
\end{definition}
\begin{remark}
	If $X_0$ is a self-adjoint element,  for $X_t$ to be self-adjoint, it is required that $a(t)$ and the sum $S=\sum\limits_{i=0}^k \int\limits_0^t b^i(s)dW_sc^i(s)$ is self-adjoint for each $t\in[0,T]$.
\end{remark}
A simple calculation shows, that the free It\^{o} formula (\ref{freeItoFormulaAsInKargin}) (in integral form see \cite[Theorem 4.1.2]{Biane1998}) implies the following $L_2(\varphi)$ isometry ($\tau<t$),
%\begin{lemma}We have
%	\begin{equation}
%		\left\| \int_\tau^t b_sdW_sc_s \right\|_2^2 = \int_\tau^t\|c_s\|_2^2 \|b_s\|_2^2 ds.
%	\end{equation}
%\end{lemma}
%
%\begin{proof}
$
		\nonumber \left\| \int_\tau^t b_sdW_sc_s \right\|_2^2 =
 \int_\tau^t \|c_s\|_2^2 \|b_s\|_2^2ds.\label{lemma:L2Abschaetzung}
$
Note that this equality implies that $\left\| \int_\tau^t b_sdW_sc_s \right\|_2=O(\sqrt{t-\tau})$.
%\end{proof}

%\textcolor{red}{This implies $a(t)^*=a(t)$. For $S^*=S$ it is sufficient that $b_i(t)^*=c_i(t),c_i(t)^*=b_i(t)$ but not necessary in the case $d>1$. Consider $d=2$, then $b_1(t)=c_1(t), b_2(t)=c_2(t)$ but also  $b_1(t)^*=c_2(t),c_1(t)^*=b_2(t)$ result in self-adjoint $S$. Combinatoric considerations lead to all possible conditions on $b_i(t),c_i(t)$ for $d>2$. }

\section{Free Stochastic Differential Equation (fSDE)}
\label{sec:fSDE}
\begin{definition} \label{def-freeItoProcess}
	Let $X_0$ be a self-adjoint element  in $\mathcal{A}^{sa}$ and $a,b^i,c^i:\mathcal{A}\rightarrow \mathcal{A}$ continuous functions in the operator norm such that $a(\mathcal A^{sa})\subset \mathcal{A}^{sa}$. We call 
	\begin{equation}\label{intro-freeSDE-diffform}
		dX_t=a(X_t)dt+ \sum\limits_{i=0}^k b^i(X_t)dW_tc^i(X_t)
	\end{equation}
	a (formal) free Stochastic Differential Equation (fSDE).	
	 A solution to \cref{intro-freeSDE-diffform} with initial condition $X(0)=X_0$ is a process $(X_t)_{t\geq 0}$ with the following properties:
	\begin{enumerate}
		\item $X(0)=X_0$ is a self-adjoint element in $\mathcal{A}_0^{sa}$
		\item $X_t\in\mathcal{A}_t^{sa}$ for all $t\geq0$
		\item The equation
		\begin{equation}\label{intro-freeSDE-inform}
			X_t=X_0 + \int_0^t a(X_s)ds + \sum\limits_{i=0}^k\int_0^tb^i(X_s)dW_sc^i(X_s)
		\end{equation} 
	    is fulfilled for all $t\geq 0$.
	\end{enumerate}
\end{definition}
\begin{remark}
	Due to the continuity of $a,b^i,c^i$ the integrals in \cref{intro-def-freeIto} are well defined. It should be noted that these function can be taken from more general spaces (see \cite{Biane1998}), but for our purposes the continuity requirement is necessary. We only consider the autonomous case, where $a,b^i,c^i$ do not explicitly depend on $t$.
\end{remark}
\begin{remark}
An existence and uniqueness theorem for fSDEs and several examples are given in \cite{kargin}. These results rely on locally operator-Lipschitz functions $a,b^i,c^i$. The existence proofs in \cite{kargin} can easily be formulated in $L_2(\varphi)$ by applying
\cref{lemma:L2Abschaetzung} instead of the free Burkholder-Gundy inequality.	
\end{remark}
As an initial example consider the free analog of the Ornstein-Uhlenbeck process (\cite{kargin}) defined by the fSDE
\begin{equation}\label{ornsteinuhlenbeck}
	dX_t=\theta X_tdt + \sigma dW_t, \, t\geq 0, \,\,\theta, \sigma\in\R.
\end{equation}
Spectral information about the solution can be obtained by taking the Cauchy transform $G$ of the self-adjoint element $X_t$.
$G$ fulfills a deterministic partial differential equation (\cite[Proposition 3.7]{kargin}). 
%\begin{displaymath}
%	G(t,z)^2+ \frac{2\theta z}{\sigma^2(e^{2\theta t}-1)} G(t,z)+\frac{2\theta z}{\sigma^2(e^{2\theta t}-1)}=0.
%\end{displaymath} 
Applying the Stieltjes inversion formula (see \cite{kargin}) to their  solution it is possible to recover the density of the distribution of $X_t$. In the case $\theta<0$ it turns out that the density of $X_t$ is a semicircle distribution with radius 
$$R=\sqrt{\frac{2\sigma^2}{|\theta|}(1-e^{-2|\theta| t})}.$$
For $t\rightarrow \infty$ the density converges to a semicircle with radius $\sigma \sqrt{\frac{2}{|\theta|}}$. The case $\theta\geq0$ is treated in the same way. For more examples we refer to \cite{kargin}.

%After introducing the main ingredients fo fSDEs we turn to the main contents of this paper.
%First, in  \cref{sec:freeItoFuncForm}, we formulate the free It\^{o}-Formula (\cite{anshelevic},\cite{Biane1998}) by applying the formalism developed in \cite{azamov_carey_dodds_sukochev_2009}. In \cite{azamov_carey_dodds_sukochev_2009} a new approach to the theory of multiple operator integrals was developed. This formalism is summarized in \cref{sec:append:OperatorInt} which summarizes the theorems from \cite{azamov_carey_dodds_sukochev_2009} applied in this paper. The new formalism allows to obtain a compact and consistent formulation of the free It\^{o} formula. Furthermore an iterated free It\^{o} formulae can be obtained out of which the free analog of the Euler-Maruyama method is motivated and derived. In \cref{sec:fEMM} additional details about the free Euler-Maruyama method are discussed. Computer implementation requires to interpret the method on the matrix level.  In \cref{sec:main} we proof the strong convergence order of $\frac{1}{2}$ and weak convergence order of $1$. Finally \cref{sec:Examples} presents examples of numerical approximations to a variety of fSDEs and a comparison the exact results whenever possible.

\section{Free It\^{o}-Formula in Functional Form}\label{sec:freeItoFuncForm}

The proof of the free It\^{o} formula in functional form \cite[Proposition 4.3.4]{Biane1998}  is done by first formulating the It\^{o} product rule for polynomials and then taking appropriate limits to operator valued functions with certain properties. Perturbation theory of operator valued functions has been intensely developed in the past decades (\cite{Skripka2019}). For functions with certain properties, which will be defined below, it is possible to give a Taylor approximation with appropriate remainder term \cite[Chapter 5.4]{Skripka2019} and derive  \cite[Proposition 4.3.4]{Biane1998} from such expansions. Let $W_n(\R)$ be  the set of functions $f\in C^n(\R)$, such that the $k$-th derivative $f^{(k)}, \, k=0,\dots,n$ is the Fourier transform of a finite measure $m_f$ on $\R$. At this point we apply the results in \cite[Corollary 5.8]{azamov_carey_dodds_sukochev_2009} which allow to apply Taylor's formula to $f\in W_n(\R)$.  Note that $f$ can be taken from more general spaces (see \cref{rem:besov}), but for our case $W_n(\R)$ is sufficient. We now derive the free It\^{o}-formula. Consider $[r,t]\subseteq\R, r\geq 0$ divided into $n$ intervals. Write
\begin{equation}
	f(X_t)-f(X_r) = \sum_{k=0}^{n-1}f(X_{i+1})-f(X_i)
\end{equation}
Applying the Taylor series expansion \cite[Corollary 5.8]{azamov_carey_dodds_sukochev_2009}, then for $f \in W_3(\mathbb{R})$ we obtain
\begin{multline}\label{eq:uuu1}
		f(X_{i+1}-X_i) = T_{f^{[1]}}^{X_i,X_i}(\Delta X,\Delta X)  + \\ + T_{f^{[2]}}^{X_i,X_i,X_i}(\Delta X,\Delta X,\Delta X) + O(\|\Delta X\|^3) 
\end{multline}
writing $\Delta X=X_{i+1}-X_i$. The definition of the multiple operator integrals $T_{f^{[1]}}, T_{f^{[2]}}$ is given in
\cite[Definition 4.1]{azamov_carey_dodds_sukochev_2009} and \cite[Lemma 4.5]{azamov_carey_dodds_sukochev_2009}.
Substituting the process $$\Delta X=\int_{t_i}^{t_{i+1}}a(X_s)ds + \int_{t_i}^{t_{i+1}}b(X_s)dW_sc(X_s)$$ and applying the product rule \cite[Theorem 4.1.2]{Biane1998} leads to a simplification of each of the integrals in \cref{eq:uuu1}. Due to the boundness of the corresponding operator valued functions $a,b,c$ with and Burkholder-Gundy inequality (\cite[Theorem 3.2.1]{Biane1998}) the necessary limits can be easily justified. Since this way of deriving the free It\^{o} formula is rather technical in notation, we do not follow this path further in detail. Just as in the classical case, in order to keep the notation as simple as possible, we do apply differential notation instead and make use of the product rules (\ref{freeItoFormulaAsInKargin}). Applying \cite[formula $(14)$ and $(15)$]{azamov_carey_dodds_sukochev_2009} and $dX_t=a_tdt + b_tdW_tc_t$ we obtain
\begin{multline}\label{eq:zzz9}
	df(X_t)=T_{f^{[1]}}^{X_0,X_0}(dX_t) - T_{f^{[2]}}^{X_t,X_0,X_0}(dX_t,dX_t) = \\ =
	T_{f^{[1]}}^{X_0,X_0}(a_tdt) + T_{f^{[1]}}^{X_0,X_0}(b_tdW_tc_t) + T_{f^{[2]}}^{X_t,X_0,X_0}(b_tdW_tc_t,b_tdW_tc_t).
\end{multline}
where $a_t=a(X_t)$ (similar for $b,c$) to shorten the notation. Applying the It\^{o} product rule (\ref{freeItoFormulaAsInKargin}) we proceed by converting each of the derivatives in \cref{eq:zzz9}. The multiple operator integrals
\begin{equation}\label{eq:xxx9}
	T_{f^{[1]}}^{X_0,X_0}(a_tdt) = \int_{\Pi^{[2]}}e^{i(s_0-s_1)X_0} a(X_t)e^{is_1X_0}d\nu_f(s_0,s_1)
\end{equation}
and
\begin{equation}\label{eq:xxxx9}
	T_{f^{[1]}}^{X_0,X_0}(b_tdW_tc_t) = \int_{\Pi^{(2)}} e^{i(s_0-s_1)X_0}\cdot \int_0^tb_sdW_sc_s \cdot e^{is_1X_0}d\nu_f(x_0,s_1).
\end{equation}
are given according to \cite[Definition 4.1]{azamov_carey_dodds_sukochev_2009} and \cite[Lemma 4.5]{azamov_carey_dodds_sukochev_2009}.
The second order derivative in \cref{eq:zzz9} can be simplified by \cref{freeItoFormulaAsInKargin} to
\begin{multline}\label{eq:zzzz9}
	T_{f^{[2]}}^{X_t,X_0,X_0}(b_tdW_tc_t,b_tdW_tc_t) = \\ =
	\int_{\Pi^{(3)}} e^{i(s_0-s_1)X_t} b_t dW_t c_t e^{i(s_1-s_2)X_0} b_t dW_t c_t e^{is_2X_0} d\nu_f(s_0,s_1,s_2) = \\ =
	\int_{\Pi^{(3)}} \varphi(c_t e^{i(s_1-s_2)X_0} b_t)    e^{i(s_0-s_1)X_t} b_t  c_t e^{is_2X_0}dt ~d\nu_f(s_0,s_1,s_2).
\end{multline}
Note that the last integral in \cref{eq:zzzz9} is no longer stochastic. Now we are able to  formulate the
\begin{theorem}[Free It\^{o} Formula in Integral Form]\label{freeItoTheorem} Suppose  $a, b, c$ are continuous functions $\mathcal{A}\rightarrow\mathcal{A}$ in the operator norm such that $a(\mathcal{A}^{sa})\subset\mathcal{A}^{sa}, b(\mathcal{A}^{sa})\subset\mathcal{A}^{sa}, c(\mathcal{A}^{sa})\subset\mathcal{A}^{sa}$. Furthermore $b,c$ are so that the product $b(X_t)dW_tc(X_t)$ is self-adjoint (resp. the sum for $k>1$). Let $(X_t)_{t\geq 0 }$ be a free It\^{o}-process and $X_0 \in \mathcal{A}^{sa}_0$ be a self-adjoint element. Then for functions $f\in W_3(\R)$ it follows that
	\begin{equation}\label{freeItoFormula}
		f(X_t)=f(X_0)+\int_0^t L^0\left[f(X_s)\right]ds + L^1\left[f(X_s)\right]_0^t
	\end{equation}
%	\begin{equation*}
%		f(X_t)=f(X_0)+\int_0^t T_{f^{[1]}}^{X_0,X_0}(X_s-X_0)ds + \frac{1}{2}\int_0^t \Delta_sf(X_s)ds
%	\end{equation*}
	where the operators $L^0,L^1:\mathcal{A}^{sa}\rightarrow\mathcal{A}^{sa}$ are introduced as an abbreviation for the expressions
	\begin{equation}\label{freeItoFormula-L0}
		L^0\left[f(X_s)\right] = T_{f^{[1]}}^{X_0,X_0}(a(X_s)) + T_{f^{[2]}}^{X_s,X_0,X_0}(b_sdW_sc_s,b_sdW_sc_s)
%		
%		 \frac{1}{2}\Delta_{b_s,c_s}f(X_s) = \\ = 
%		 T_{f^{[1]}}^{X_0,X_0}(a(X_s)) + 
%		\int_{\Pi^{(3)}} \varphi(c_s e^{i(s_1-s_2)X_0} b_s)    e^{i(s_0-s_1)X_s} b_s  c_s e^{is_2X_0} ds ~d\nu_f(s_0,s_1,s_2) 
	\end{equation}
and
	\begin{equation}\label{freeItoFormula-L1}
		L^1[f(X_s)]_0^t = T_{f^{[1]}}^{X_0,X_0}(b_sdW_sc_s).
		%\label{eq:hh2}
	\end{equation}
The operator integrals are given by \cref{eq:xxx9,eq:xxxx9,eq:zzzz9}.
\end{theorem}
%\begin{remark}\label{rem:besov} We can transform each integral in \cref{freeItoFormula} into same formalism as in \cite[Proposition 4.3.4]{Biane1998}. For example consider
%	\begin{multline*}
%		\int_{\Pi^{(1)}} e^{i(s_0-s_1)X_0}a(X_s)e^{i s_1X_0}d\nu_f^{(1)}(s_0,s_1) =\\
%		= \int_{\Pi^{(1)}} e^{is_0(1-\frac{s_1}{s_0})X_0}a(X_s)e^{i s_0\frac{s_1}{s_0}X_0}d\nu_f^{(1)}(s_0,s_1)=\\
%		=\int_0^{1} \int_{\R} \frac{i~s_0}{\sqrt{2\pi }} e^{is_0(1-\alpha)X_0}a(X_s)e^{i s_0\alpha X_0}dm_f^{(1)}(s_0) d\alpha
%	\end{multline*}
%\end{remark}
\begin{remark}\label{rem:besov}
	The function $f$ in \cref{freeItoTheorem} can be taken from the Besov space $\mathcal{B}_{\infty 1}^{n}(\R)$ for which  $W_{n}(\R)\subset \mathcal{B}_{\infty 1}^{n}(\R)$. For a definition of $\mathcal{B}_{\infty 1}^{n}(\R)$ we refer to \cite[pp. 9]{Skripka2019}. For the purpose of this paper it is sufficient to consider $W_n(\R)$.
\end{remark}
%\begin{remark} Applying the trace to (\ref{freeItoFormula}) yields
%	\begin{equation}
%		\varphi(f(X_t)) = \varphi(f(X_0)) + \int_0^t \varphi \left( L^0\left[f(X_s)\right] \right) ds.
%	\end{equation}
%\end{remark}
%\begin{proof}Alternative proof of \cref{freeItoTheorem}. 
%\end{proof}
\section{Free analog of Euler-Maruyama Method (fEMM)}
\label{sec:fEMM}
We are now going to define a method for the numerical solution of the fSDE  (\ref{intro-freeSDE-diffform}). For simplicity we assume $d=1$ in the following. Consider the free It\^{o} process (\ref{intro-def-freeIto}) over the time interval of length $\Delta t$,
\begin{equation}\label{eq:h123}
	X_{t+\Delta t}=X_t + \int_t^{t+\Delta t}a(X_s)ds +\int_t^{t+\Delta t}b(X_s)dW_sc(X_s).
\end{equation}
Assuming $a,b,c\in W_3(\R)$ we can apply the free It\^{o} Formula (\ref{freeItoFormula}) for $f=a,b,c$ in \cref{eq:h123}. This yields an iterated free It\^{o} formula which allows to motivate and define a free analog of the Euler-Maruyama method. Using the abbreviations $a(X_t)=a_t$ (similar notation for $b,c$) and $t_1=t+\Delta t$ we obtain
\begin{multline}\label{eq:h1231}
	X_{t_1}-X_t=\int_t^{t_1} a_t ds+
	\int_t^{t_1}\int_t^s L^0[a_u]du\,ds+\int_t^{t_1}L^1[a_u]_t^{s}ds + \\
	\int_t^{t_1}\left\{\left( b_t+\int_t^sL^0[b_u]du+L^1[b_u]_t^{s}) \right)dW_s\left( c_t+\int_t^sL^0[c_u]du+L^1[c_u]_t^{s}) \right)\right\}
\end{multline}
Since $a_t,b_t,c_t$ do not depend on the integration variable $s$, we rewrite \cref{eq:h1231} as
\begin{equation}\label{eq:h4}
	X_{t_1}-X_t=a_t\Delta t+b_t(W_{t_1}-W_t)c_t + \rho,
\end{equation}
with the remainder
\begin{multline}\label{eq:h5}
	\rho = \int_t^{t_1}\int_t^s L^0[a_u]du\,ds+ \int_t^{t_1}L^1[a_u]_t^sds + \\
	+\int_t^{t_1} b_t dW_s\left( \int_t^sL^0[c_u]du+L^1[c_u]_t^{s} \right)+\\
	+\int_t^{t_1}\left(\int_t^sL^0[b_u]du\right) dW_s\left( c_t+\int_t^sL^0[c_u]du+L^1[c_u]_t^{s} \right)+\\
	+\int_t^{t_1}\left(L^1[b_u]_t^{s}\right)dW_s\left( c_t+\int_t^sL^0[c_u]du+L^1[c_u]_t^{s} \right).
\end{multline}
By the boundedness and continuity  of the involved functions $a,b,c$ the above integrals are well defined.
In the case $d>1$ we can simply put the sum-sign in front of each integral in $\rho$ which contains either $b$ or $c$. The free Euler-Maruyama method can now be motivated from \cref{eq:h4} by simply skipping the remainder $\rho$.
\begin{definition}[fEMM]\label{freeEM-Definition}
	Given $T>0$, consider a partition of $[0,T]$ into $L\in\N$ intervals $[t_{k-1},t_{k}],k=1,\dots,L$ with constant step size $\Delta t=\frac{T}{L}$.
	 Define the one-step free Euler-Maruyama approximation (fEMM) $\oX_k$ of the solution $X_{t}$ of \cref{intro-freeSDE-diffform} on $[0,T]$ by
	\begin{equation}\label{kap2-def-freeEM}
		\oX_{k+1}=\oX_{k}+a(\oX_{k})\Delta t+b(\oX_{k})\Delta W_{k}c(\oX_{k}),\,\,\, 	k=0,1,\dots,L-1
	\end{equation}
	with starting value $X_0=\oX_0\in\mathcal{A}^{sa}$ and $\Delta W_{k}=W_{k+1}-W_{k}$. $\oX_k$ denotes the numerical approximation to $X_t$ at timepoint $t_k$. 
\end{definition}
In general the fSDE and fEMM act in a finite unital faithful von Neumann algebra $\mathcal{A}$. For the implementation on a computer it is necessary to consider fEMM in the von Neumann algebra $\mathcal{M}_N^{sa}(\R)$ of random matrices (\cref{sec:FreeProbaTheo}). This leads to the situation shown in \cref{fig:freeEM-diagramm}. Given a solution $X_t\in\mathcal{A}^{sa}$ of the fSDE (\ref{intro-freeSDE-diffform}) at $t=k\Delta t\in[0,T],\,k\in\{0,\dots,L-1\}$. Applying fEMM in $\mathcal{A}^{sa}$ we obtain an approximation $\oX_k\in\mathcal{A}^{sa}$. Considering the implementation of fEMM on a computer, we obtain an approximation $\oX_k^N\in\mathcal{M}_N^{sa}(\R)$ to $\oX_k$. To judge the quality of the approximation we have to consider two limits, one by the dimension $N$ of the random matrix in $\mathcal{M}_N^{sa}(\R)$, the other by the step size $\Delta t \rightarrow 0$.
\begin{figure}[ht]
	\centering
	\begin{tikzpicture}
		\matrix (m) [matrix of math nodes,row sep=6em,column sep=6em,minimum width=4em]
		{
			X_{t}\in \mathcal{A}^{sa} & \oX_k \in \mathcal{A}^{sa} \\
			X_{t}^N \in \mathcal{M}_N^{sa}(\R) & \oX_k^N \in \mathcal{M}_N^{sa}(\R) \\};
		\path[-stealth]
		(m-2-1) edge node [left] {$N\rightarrow \infty$} (m-1-1)
		(m-2-2) edge node [right] {$N\rightarrow \infty$} (m-1-2)
		(m-2-2) edge node [below] {$\Delta t \rightarrow 0$} (m-2-1)
		edge [dashed,->] (m-1-1)
		(m-1-2) edge node [above] {$\Delta t \rightarrow 0$} (m-1-1);
	\end{tikzpicture}
	\caption{Diagram of the approximation scheme of fEMM. Implementation of \ref{freeEM-Definition} is realized in $\mathcal{M}_N^{sa}(\R)$ (bottom right), which gives an approximation to the solution $X_t$ of \cref{intro-freeSDE-diffform} (top left). The limits $N\rightarrow\infty$ according to the size of random matrices and the step size limit $\Delta t$ do commute and give convergence of $\oX_k^N$ in distribution to $X_t\in\mathcal{A}^{sa}$.}
	\label{fig:freeEM-diagramm}
\end{figure}
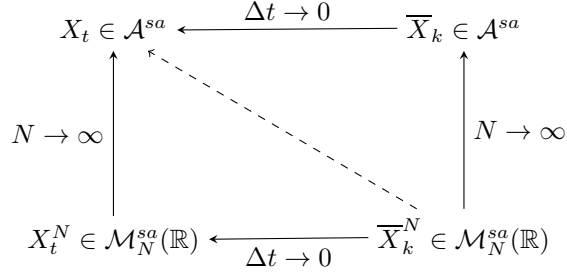

Fix $N \in \N$, $k\in\{0,\dots,L-1\}$. Consider the element $\oX_k^N$ in \cref{fig:freeEM-diagramm} (right bottom). Due to the strong convergence of fEMM (\cref{thm1}) we deduce the existence of the element $X_t^N\in\mathcal{M}_N^{sa}(\R)$ at $t=k\Delta t$ and with \cref{lemma:ChaucyPointwise} the strong convergence implies convergence in distribution to $X_t^N$ which is a solution of \cref{intro-freeSDE-diffform} in $\mathcal{M}_N^{sa}(\R)$.
For each $t\in[0,T]$ we get a sequence $(X_t^N)_N$ in $\mathcal{M}_N^{sa}(\R)$.
Convergence to an element $X_t\in\mathcal{A}^{sa}$ follows by \cite[Exercise 25]{TaoBlog}. Due to \cite[Theorem 4.4.1]{voicudykemanica} free stochastic calculus, stochastic integrals and the free It\^{o}-formula can be viewed as large $N$  limit of stochastic calculus with respect to $N\times N$ hermitian matrices. Then  limit $N\rightarrow\infty$ converge in distribution to the solution of \cref{intro-freeSDE-diffform}. By the same arguments we can first take the large $N$ limit first followed by $\Delta t \rightarrow 0$. 
%\textcolor{red}{
%Taking large $N$ limit we want to have $X_t^N\rightarrow X_t\in\mathcal{A}$ in distribution.
%As in Speicher \cite{Biane1998}, the Free Brownian motion, the Stochastic integral and the Free It\^{o}-formulae (\ref{freeItoFormulaAsInSpeicher}) can be viewed as the large $N$ limit from $\mathcal{M}_N(\R)$ to $\mathcal{A}$. First
%$W_t=\lim\limits_{N\rightarrow \infty}W_t^N$. Then $\int b_s^NdW^N_sc_s^N \rightarrow \int b_sdW_sc_s$ and $\int a(X_s^N)ds\rightarrow \int a(X_s)ds$.
%By taking $f(x)=x$ it follows, that $X_t$ can be viewed as the large $N$ limit of $X_t^N$. Then $G^N(t,z)=\varphi((X_t^N-z)^{-1})$ converges to $G(t,z)=\varphi((X_t-z)^{-1})$.}
\section{Convergence Results}
\label{sec:main}
This sections gives two theorems regarding strong and weak convergence properties of fEMM. The results will be numerically verified in \cref{sec:Examples}.
%\textcolor{red}{Wir haben Konvergenzordnung bzgl. $\Delta t$. Brauchen wir auch Konvergenzordnung bezüglich $N$ ?}
\subsection{Strong convergence of fEMM}\label{subsec:strongconverg}
\begin{definition}\label{def-strong-converg-fEMM}
	The numerical approximation fEMM (\cref{freeEM-Definition}) is said to converge strongly to the solution $X_t$ of \cref{intro-freeSDE-diffform} with order $p>0$, if there is a constant $C>0$ independent of $\Delta t$, so that
	\begin{equation}
		\sup_{0\leq t_k\leq T}\varphi\left(\left|\oX_{k}-X_k\right|\right)\leq C (\Delta t)^p.
	\end{equation}
	for any fixed time point $t_k=k\Delta t\in [0,T], \, k=0,\dots,L$. $X_k$ denotes the solution $X_t$ evaluated at $t_k$ and $X_k=X(t_k)$. At $t=0$ we have $X(0)=X_0=\oX_0$.
\end{definition}
\begin{theorem}\label{thm1}
	Consider the fSDE (\ref{intro-freeSDE-diffform}) and $L_a>0$. Let $a:\mathcal{A}\rightarrow \mathcal{A}$ be an operator function with $a(\mathcal A^{sa})\subset \mathcal A^{sa}$. Additionally let the function $a$ be operator Lipschitz in $L_2(\varphi)$, i.e.
		\begin{displaymath}
			\left\| a(X)-a(Y)\right\|_2 \leq L_a\|X-Y\|_2,
			%\left\|a(X)\right\|_2\leq L_a(1+\left\|X\right\|_2)
		\end{displaymath} for arbitrary elements $X,Y\in\mathcal{A}^{sa}$.
 Analog conditions hold for functions $b$ and $c$. Then the fEMM approximation (\ref{freeEM-Definition}) has strong convergence order of $p=\frac{1}{2}$, i.e.
	\begin{equation}
		\sup_{0\leq t_k\leq T}\E\left(\left|\oX_k-X_k\right|\right)\leq C (\Delta t)^\frac{1}{2}.
	\end{equation}
	The constant $C$ is independent of step size $\Delta t$.
\end{theorem}
\begin{remark} In section \cref{sec:fEMM} we mentioned, that for the implementation on a computer we use fEMM in $\mathcal{A}^{sa}=\mathcal{M}_N^{sa}(\R)$. The definition of fEMM \cref{freeEM-Definition}, the strong convergence property of $p=\frac{1}{2}$ and weak order of convergence $p=1$ (see section \cref{subsec:weakconvergence}) can be directly carried over to the von Neumann algebra of $N\times N$ random matrices $\mathcal{M}^{sa}_N(\R)$.
\end{remark}
The proof of \cref{thm1} closely follows the proof of strong convergence of the Euler-Maruyama method for commutative stochastic differential equations, see \cite{HighamKloeden2020}. The differences lie in estimating the free stochastic integrals in $L_2(\varphi)$ (see \cref{lemma:L2Abschaetzung}).
\begin{proof}[Proof of \cref{thm1}]
From the fEMM approximation $\oX_k$ at the time point $t_k, k=1,\dots L$ we define a step process $\oX(t)=\oX_k$ for $t_{k-1}\leq t < t_{k}$. We use the short notion 
\begin{displaymath}
	\oX(s)=\oX_s, X(s)=X_s, a(X(s))=a_s, a(\oX(s))=\oa_s.
\end{displaymath}
Analog for $b,c$. Consider a point $t\in[0,T]$. Let $n_t\in\N$ such that $t\in[t_{n_t},t_{n_t+1}[$. Then 
\begin{multline}\label{eq:help1}
  \oX_t-X_t=\oX_{n_t}-X_t=
	\oX_{n_t}-\left(X_0+\int_0^t a(X_s)ds + \int_0^t b(X_s)dW_sc(X_s)\right)=\\
	=\sum_{k=0}^{n_t-1}(\oX_{k+1}-\oX_k)-\int_0^t a_sds - \int_0^t b_sdW_sc_s=\\
	=\sum_{k=0}^{n_t-1}\oa_k\Delta t + \sum_{k=0}^{n_t-1}\ob_k\Delta W_k \oc_k -\int_0^t a_sds - \sum_{k=0}^{n_t-1} \int_{t_k}^{t_{k+1}} b_s dW_s c_s
\end{multline}
Due to the definition of the step-wise process $\oX(t)$ we can reformulate the terms $\oa_k\Delta t$ and $\ob_kdW_s\oc_k\Delta t$ as an integrals as follows. We deduce
\begin{displaymath}
	\oa_k\Delta t = a(\oX_k) \Delta t = a(\oX(t_k)) \Delta t = \int_{t_k}^{t_{k+1}}a(\oX(s))ds=
	\int_{\Delta t} \oa_s ds
\end{displaymath}
and 
\begin{displaymath}
	\ob_kdW_s\oc_k\Delta t = \int_{t_k}^{t_{k+1}}b(\oX(s))dW_sc(\oX(s))ds=
	\int_{\Delta t}\ob_sdW_s\oc_s.
\end{displaymath}
Note that $\oa_s,\ob_s,\oc_s$ are constant over $[t_k,t_{k+1}[$. Continuing from the last line of \cref{eq:help1} we obtain
\begin{displaymath}
		\oX_t-X_t=\int_0^{t_{n_t}}\oa_sds + \int_0^{t_{n_t}} \ob_sdW_s\oc_s
		-\int_0^t a_sds - \int_0^t b_sdW_sc_s
\end{displaymath}
and further
\begin{multline*}
	\oX_t-X_t=\int_0^{t_{n_t}}\left(\oa_s-a_s\right)ds-\int_{t_{n_t}}^ta_sds+\\
	+\int_0^{t_{n_t}} \ob_sdW_s\oc_s - 	\int_0^{t_{n_t}}b_sdW_sc_s
	- \int_{t_{n_t}}^t b_sdW_sc_s.
	\end{multline*}
Then the square of the $L_2(\varphi)$-norm of the difference $\oX_t-X_t$ is
\begin{multline}\label{ineq:hh1}
	\varphi\left(\left|\oX_t-X_t\right|^2\right)= \varphi\left(\left|\int_0^{t_{n_t}}\left(\oa_s-a_s\right)ds-\int_{t_{n_t}}^ta_sds\right.\right.+\\\left.\left.
	+\int_0^{t_{n_t}} \ob_sdW_s\oc_s
	-\int_0^{t_{n_t}}b_sdW_sc_s
	-\int_{t_{n_t}}^t b_sdW_sc_s\right|^2\right).
\end{multline}
By applying the inequality
\begin{displaymath}
	\|X_1+X_2+X_3+X_4\|_2^2\leq 4\left(\|X_1\|_2^2+\|X_2\|_2^2+\|X_3\|_2^2+\|X_4\|_2^2\right)	\end{displaymath}
for $X_1,X_2,X_3,X_4\in\mathcal{A}^{sa}$, we deduce from \cref{ineq:hh1}
%From \cref{lemma:hilf3} we deduce
\begin{multline}\label{eq:hilf66}
	\varphi\left(\left|\oX_t-X_t\right|^2\right)\leq  4\varphi\left(\left|\int_0^{t_{n_t}}\left(\oa_s-a_s\right)ds\right|^2\right)+
	4\varphi\left(\left|\int_{t_{n_t}}^ta_sds\right|^2\right)+\\
	4\varphi\left(\left|\int_0^{t_{n_t}} \ob_sdW_s\oc_s -  \int_0^{t_{n_t}}b_sdW_sc_s\right|^2\right)+
	4\varphi\left(\left|\int_{t_{n_t}}^t b_sdW_sc_s\right|^2\right).
\end{multline}
%The Cauchy-Schwarz inequality implies
% ( siehe auch (10.18) Seite 107 in \cite{HighamKloeden2020})
Using the abbreviation
\begin{displaymath}
	%v(t)=\varphi\left(\left|\oX_t-X_t\right|^2\right),
	v(t)=\left\|\oX_t-X_t\right\|_2^2
\end{displaymath}
and applying Jensen's inequality it follows from \cref{eq:hilf66} that
\begin{multline}\label{eq:help55}
	v(t)\leq  4T\int_0^{t_{n_t}}\varphi\left(\left|\left(\oa_s-a_s\right)\right|^2\right)ds+
	4\Delta t \int_{t_{n_t}}^t\varphi\left(\left|a_s\right|^2\right)ds+\\
	4\varphi\left(\left|\int_0^{t_{n_t}} \ob_sdW_s\oc_s - \int_0^{t_{n_t}}b_sdW_sc_s\right|^2\right)+
	4\varphi\left(\left|\int_{t_{n_t}}^t b_sdW_sc_s\right|^2\right).
\end{multline}
Estimating the first integral in (\ref{eq:help55}) gives
\begin{multline}\label{eq:firstintegral}
	\int_0^{t_{n_t}}\varphi\left(\left|\left(\oa_s-a_s\right)\right|^2\right)ds =
	\int_0^{t_{n_t}}\|\oa_s-a_s\|_2^2ds \leq 
	\int_0^{t_{n_t}} L_a^2 \|\oX_s-X_s\|_2^2ds =\\
	= L_a^2\int_0^{t_{n_t}} \varphi\left(\left|\oX_s-X_s\right|^2\right)ds
	=L_a^2\int_0^{t_{n_t}} v(s) ds.
\end{multline}
The second integral in (\ref{eq:help55}) is an $O(\Delta t)$, since
\begin{multline}\label{eq:secondintegral}
	 \int_{t_{n_t}}^t\varphi\left(\left|a_s\right|^2\right)ds =
	 \int_{t_{n_t}}^t \|a_s\|_2^2 ds \leq C_a\int_{t_{n_t}}^t (1+\|X_s\|_2^2) ds\leq \\ \leq C_1(t-t_{n_t}) \leq C_1 \Delta t.
\end{multline}
The constant $C_1\in\R$ does not depend on $\Delta t$. The third integral in (\ref{eq:help55}) is estimated as follows.
\begin{multline}\label{eq:thirdintegral}
	\varphi\left(\left|\int_0^{t_{n_t}} \ob_sdW_s\oc_s -  \int_0^{t_{n_t}}b_sdW_sc_s\right|^2\right)=\\
%	=\varphi\left(\left|\int_0^{t_{n_t}} (\ob_s-b_s)dW_s\oc_s +  \int_0^{t_{n_t}}b_sdW_s(\oc_s-c_s)\right|^2\right) \leq  \\	
	\leq \varphi\left(\left|\int_0^{t_{n_t}} (\ob_s-b_s)dW_s\oc_s\right|^2 +  \left|\int_0^{t_{n_t}}b_sdW_s(\oc_s-c_s)\right|^2\right) =\\
%	= \varphi\left(\left|\int_0^{t_{n_t}} (\ob_s-b_s)dW_s\oc_s\right|^2\right) +   \varphi\left(\left|\int_0^{t_{n_t}}b_sdW_s(\oc_s-c_s)\right|^2\right) = \\
	=\left\| \int_0^{t_{n_t}} (\ob_s-b_s)dW_s\oc_s \right\|_2^2+\left\|\int_0^{t_{n_t}}b_sdW_s(\oc_s-c_s)\right\|_2^2 
\end{multline}
Applying to the $L_2(\varphi)$ isometry of the stochastic integral, the Lipschitz conditions on $a,b,c$ and the Cauchy-Schwarz inequality we continue from the last line of \cref{eq:thirdintegral} to get

\begin{equation}
	\varphi\left(\left|\int_0^{t_{n_t}} \ob_sdW_s\oc_s -  \int_0^{t_{n_t}}b_sdW_sc_s\right|^2\right) \leq K\int_0^{t_{n_t}} \|\oX(s)-X(s)\|_2^2 ds = K\int_0^{t_{n_t}} v(s)ds.
\end{equation}
%\begin{multline}
%	\leq 
%	\int_0^{t_{n_t}} \|  \ob_s-b_s \|_2^2\|\oc_s\|_2^2 ds +  \int_0^{t_{n_t}} \|  \oc_s-c_s \|_2^2\|b_s\|_2^2 \leq \\
%	\leq  \int_0^{t_{n_t}} K_1\|\oX(s)-X(s)\|_2^2\left(1+\left\|\oX(s)\right\|_2^2\right)^2ds+\\
%	+ \int_0^{t_{n_t}} K_2\|\oX(s)-X(s)\|_2^2\left(1+\left\|X(s)\right\|_2^2\right)^2ds \leq \\
%	\leq K\int_0^{t_{n_t}} \|\oX(s)-X(s)\|_2^2 ds = K\int_0^{t_{n_t}} v(s)ds
%\end{multline}
The constant $K$ depends on the Lipschitz constants $L_b, L_c$ and the $L_2(\varphi)$ norm of $\|X(s))\|_2^2$ and $\|\oX(s))\|_2^2$ which are uniformly bounded since $X(s)\in\mathcal{A}^{sa}$ and $\oX(s)\in\mathcal{A}^{sa}$. The constant $K$ does not depend on $\Delta t$. The last stochastic integral in (\ref{eq:help55}) is handled by the $L_2(\varphi)$ isometry of the stochastic integral, i.e.
\begin{multline}\label{eq:fourthintegral}
	\varphi\left(\left|\int_{t_{n_t}}^t b_sdW_sc_s\right|^2\right) = 
	\int_{t_{n_t}}^t \|b_s\|_2^2\|c_s\|_2^2ds \leq \\
	\leq	C_3\int_{t_{n_t}}^t \left(1+\left\|X(s)\right\|_2^2\right)^2 ds \leq 
	C_4(t_{n_t}-t)\leq C_4\Delta t.
\end{multline}
Again, since $\sup\limits_{s\in[0,T]}\|X(s)\|_2<\infty$ by definition, we have $C_4<\infty$ and does not depend on $\Delta t$.
Inserting (\ref{eq:firstintegral}), (\ref{eq:secondintegral}), (\ref{eq:thirdintegral}), (\ref{eq:fourthintegral}) into (\ref{eq:help55}) yields
\begin{displaymath}
	v(t) \leq 4(TL_a^2+K)\int_0^{t_{n_t}}v(s)ds+4C_1 L_a\Delta t^2+4C_4\Delta t
\end{displaymath}
For $\Delta t$ small enough $v(t)$ fulfills the inequality
\begin{displaymath}
	v(t) \leq D\Delta t + E\int_0^{t_{n_t}}v(s)ds.
\end{displaymath}
The Gronwall inequality implies
\begin{displaymath}
	 v(t)\leq F \Delta t, F<\infty,\, t\in[0,T].
\end{displaymath}
The supremum of the $L_1(\varphi)$-norm over $[0,T]$ of the error $\overline{X}(t)-X(t)$ is first estimated by
\begin{displaymath}
	\sup_{0\leq s\leq T}\varphi(|\oX(s)-X(s)|) \leq \sup_{0\leq s\leq T}\varphi(|\oX(s)-X(s)|^2)^{\frac{1}{2}} \leq \sqrt{F}\sqrt{\Delta t}.
\end{displaymath}
Since $\oX(t_k)=\oX_k$ for all $0\leq t_k\leq T$ we have
\begin{displaymath}
	\sup_{0\leq t_k\leq T}\varphi(|\oX_k-X_k|) \leq C\sqrt{\Delta t}.
\end{displaymath}
\end{proof}

\subsection{Weak convergence of fEMM}\label{subsec:weakconvergence}

The main content of this section is \cref{weak-main-thm}, which states weak convergence of order $p=1$ under cetain assumption on the coefficient functions $a,b,c$. First we give the definition of weak convergence in the context of fSDEs. To prove \cref{weak-main-thm} we need some preparatory statements. At first, \cref{lem:weak:onestep} states that the expectation of the remainder in \cref{eq:h5} of the iterated It\^{o} formula (\ref{eq:h4}) is  $O(\Delta t^2)$. This allows to formulate  \cref{thm:weakonestep}, which states weak order $p=2$ for one single fEMM step. It is then possible to take over the proof in \cite[Theorem 2.2.1]{milsteintretyakov} to obtain the desired result of weak convergence order $p=1$. \\
In the sequel, we use the abbreviations $\Delta = X_{t+\Delta t} - X_t$ and $\oDelta=\overline{X}_{t+\Delta t}-X_t$. Note that  
$
	\rho = X_{t+\Delta t}-\overline{X}_{t+\Delta t} = 
	(X_{t+\Delta t}-X_t)-(\overline{X}_{t+\Delta t}-X_t) =\Delta-\oDelta.
$
\begin{definition}\label{dek:weakconv}
	The numerical approximation fEMM defined by (\ref{kap2-def-freeEM}) is said to converge weakly to the solution $X_t$ of (\ref{intro-freeSDE-diffform}) with order $p>0$, if there is a constant $C_{f}>0$ independent of $\Delta t$, so that for $f$ from a sufficiently large class of functions
	\begin{equation}
		\sup_{0\leq t_k\leq T}\left| \varphi \left(f\left(X_{k} \right)\right) - \varphi\left(f\left(  \oX_k\right)\right)\right|\leq C_{f,T} (\Delta t)^p
	\end{equation}
	as $\Delta t \rightarrow 0$.
\end{definition}
We start with
\begin{lemma}\label{lem:weak:onestep} Consider the free It\^{o} process (\ref{intro-def-freeIto}) over the time interval $[t,t+\Delta t]$. Let $a,b,c\in W_3(\R)$ and uniformly bounded in $\mathcal{A}$. Then there is a constant $K>0$ independent of $\Delta t$ such that the following inequality holds,
\begin{equation}\label{eq:statement:h1}
	\left|\E(\rho)\right| =\left|\E(\Delta - \oDelta)\right|\leq K\Delta t^2.
\end{equation}
\end{lemma}
\begin{proof}
	Applying the trace $\varphi$ to the iterated It\^{o} formula (\ref{eq:h5}) we have to consider in total 10 integrals (by resolving the brackets). 
	All integrals in (\ref{eq:h5}) except ($t_1=t+\Delta t$)
	\begin{equation}
		\varphi(\rho)=\varphi\left(\int_t^{t_1}\int_t^s L^0[a_u]du\,ds\right)
	\end{equation}
	are zero due to freeness property of the free Brownian motion and zero trace of the stochastic integral.
We start using  \cref{freeItoFormula-L0} and the definition of the multiple operator integrals \cref{eq:xxx9,eq:xxxx9} and proceed as 
\begin{multline}\label{eq:hhh6}
		\left|\varphi\left(\int_t^{t_1}\int_t^s L^0[a_u]du\,ds\right)\right| \leq \\
		\leq \left|\varphi\left(\int_t^{t_1}\int_t^s 
		\int_\Pi e^{i(s_0-s_1)X_t}a(X_u)e^{is_1X_t}d\nu_a(s_0,s_1) du\,ds\right) \right| + \\
		 + \left| \varphi \left(\int_\Pi e^{i(s_0-s_1)X_t}b_uc_u\varphi(c_ub_u)e^{is_2X_t}d\nu_a(s_0,s_1,s_2)
		 du\,ds\right)\right|=I_1+I_2.
		%		\leq \varphi\left(\int_t^{t_1}\int_t^s |a_u| + |b_u||c_u| du\,ds\right)\leq K_7 \Delta t^2,
\end{multline}
Due to the freeness of the factors in the integrand of $I_1$ we obtain the estimation
\begin{multline*}
		I_1\leq \int_t^{t_1}\int_t^s 
		\int_\Pi \left|\varphi\left( e^{i(s_0-s_1)X_t}a(X_u)e^{is_1X_t}\right) \right|d\nu_a(s_0,s_1)du\,ds = \\ =
		\int_t^{t_1}\int_t^s 
		\int_\Pi \left|\varphi\left( e^{i(s_0-s_1)X_t}\right)\varphi\left(a(X_u)\right)\varphi\left(e^{is_1X_t}\right) \right|d\nu_a(s_0,s_1)du\,ds = \\ =
		\int_t^{t_1}\int_t^s 
		\int_\Pi \left|\varphi\left(a(X_u)\right) \right|d\nu_a(s_0,s_1) \leq K_1\Delta t^2.
\end{multline*}
where $K_1$ independent of $\Delta t$. The last equality follows because $\left|\varphi\left(e^{i(s_0-s_1)X_t}\right)\right|=1$ and $a(X_u)$ is uniformly bounded in $\mathcal{A}$. Furthermore  $(\Pi^{(2)},\nu_a)$ is a finite measure space (\cite{azamov_carey_dodds_sukochev_2009}). The 
second integral $I_2$ in the last line of \cref{eq:hhh6} is estimated as
\begin{multline}\label{ineq:help2}
	\left|\varphi\left(
		\int_t^{t_1} \int_t^s\int_\Pi e^{i(s_0-s_1)X_t}b_uc_u\varphi(c_ub_u)e^{is_2X_t}d\nu_a(s_0,s_1,s_2)du\,ds\right) \right| \leq \\ \leq 
				\int_\Pi \varphi^2(b_uc_u)d\nu_a(s_0,s_1,s_2) \leq K_2\Delta t^2
\end{multline}
The last inequality follows, since $b_u,c_u$ are uniformly bounded in $\mathcal{A}$.
\end{proof}
Now we are fully prepared to formulate and prove
\begin{theorem}\label{thm:weakonestep} Consider one single step of fEMM (see \cref{def}) with start value $X_t\in\mathcal{A}_t^{sa}$ at time point $t\in[0,T]$. Let  $f\in W_2(\R)$. If $a,b,c \in W_3(\R)$ and uniformly bounded in $\mathcal{A}$, then one single fEMM step of size $\Delta t$ with starting value $X_t$ has  weak convergence of order $p=2$,  i.e.
\begin{equation}
		|\varphi(f(X_{t+\Delta t}))  - \varphi(f(\overline{X}_{t+\Delta t}))| \leq K \Delta t^2.
\end{equation}
\end{theorem}
\begin{proof} $\Delta = X_{t+\Delta t}-X_t$. $\overline{\Delta} = \overline{X}_{t+\Delta t}-X_t$
	According to \cite[Corollary 5.8]{azamov_carey_dodds_sukochev_2009} we develop $f$ into a Taylor Series for $\Delta$, resp. $\overline{\Delta}$. Let $f\in W_2(\R)$, then 
\begin{displaymath}
		f(X_{t+\Delta t})=f(\Delta + X_t)=f(X_t) + T_{f^{[1]}}^{X_t,X_t}(\Delta) +  R_{\Delta}
\end{displaymath}
and
\begin{displaymath}
		f(\oX_{t+\Delta t})=f(\oDelta + X_t)=f(X_t) + T_{f^{[1]}}^{X_t,X_t}(\oDelta) +  R_{\oDelta}.
\end{displaymath}
Substracting yields
\begin{equation*}
		f(X_{t+\Delta t})  - f(\overline{X}_{t+\Delta t}) = 	T_{f^{[1]}}^{X_t,X_t}(\Delta-\oDelta)+ R_{\Delta}-R_{\oDelta}.
\end{equation*}
For the remainder we choose the integral form (see \cite[Theorem 1.43]{schwartzNonLinFANA1696})
\begin{displaymath}
		R_\Delta=\frac{1}{2}\int_0^1(1-\tau) T_{f^{[2]}}^{X_t+\tau\Delta, X_t+\tau\Delta}(\Delta, \Delta) d\tau
\end{displaymath}
and analog for $\oDelta$.
	%\begin{displaymath}
	%	R_{\oDelta}=\frac{1}{2}\int_0^1(1-\tau) T_{f^{[2]}}^{X_t+\tau\oDelta, X_t+\tau\oDelta}(\oDelta, \oDelta) d\tau.	
	%\end{displaymath}
Applying the definition of multiple operator integrals (see \cite[Lemma 4.5]{azamov_carey_dodds_sukochev_2009}) and the trace $\varphi$ yields 
\begin{multline}\label{thm:weak:onestep:proof:1stestim}
	\left| \varphi\left(T_{f^{[1]}}^{X_t,X_t}(\Delta-\oDelta)\right)\right|~ = \\
	=\left|\int_{\Pi^{[2]}} \varphi\left( e^{i(s_0-s_1)X_t}(\Delta-\oDelta)e^{i(s_1-s_2)X_t}\right)d\nu(s_0,s_1)\right|= \\
	\left|
	\int_{\Pi^{[2]}} \varphi\left( e^{i(s_0-s_2)X_t}\right)\varphi\left(\Delta-\oDelta\right)d\nu(s_0,s_1)\right| \leq \int_{\Pi^{[2]}} \left|\varphi\left(\Delta-\oDelta\right)\right|d\nu(s_0,s_1) \leq \\ \leq K\Delta t^2 \|m_{f^{(1)}}\| = K_3\Delta t^2,
\end{multline}
due to freeness of the factors in the integrand. The last line follows by \cref{lem:weak:onestep} (for $f\in W_n(\R)$ the measure $m_{f^{(1)}}$ is finite). We turn to the remainder $R_\Delta$. 
\begin{multline}\label{thm:weak:onestep:proof:2ndtestim}
	\left| \varphi\left(R_{\Delta}\right) \right|
	=\left| \int_0^1(1-\tau) \varphi\left(T_{f^{[2]}}^{X_t+\tau\Delta, X_t+\tau\Delta}(\Delta, \Delta) \right)d\tau\right| =\\	
	=\left| \int_0^1(1-\tau) \int_{\Pi^{(3)}} \varphi (e^{i(s_0-s_1)(X_t+\tau\Delta)}) \varphi(\Delta) \right. ...\\ \left. ...\varphi (e^{i(s_1-s_2)(X_t+\tau\Delta)}) \varphi(\Delta)\varphi (e^{i(s_2)(X_t+\tau\Delta)})d\nu_f(s_0,s_1,s_2) d\tau\right| \leq \\	
	\leq C_5\int_0^1(1-\tau) \left| \varphi(\Delta)^2\right|d\tau \leq C_6 \left| \varphi(\Delta)^2\right| = \\
	= C_6\left| \varphi\left(\int_{\Delta t} a_sd_s + \int_{\Delta t}b_sdW_sc_s\right)\right|^2 = C_6 \left| \varphi\left(\int_{\Delta t}a_sds\right) \right|^2 \leq \\
	\leq C_7 \varphi\left(\int_{\Delta t}\left|a_s\right|ds\right)^2 = 
	C_7 \left(\int_{\Delta t}\varphi\left(\left|a_s\right|\right)ds\right)^2 \leq C_8 \Delta t^2
\end{multline}
In similar consideration it follows that 
\begin{equation}\label{thm:weak:onestep:proof:3rdtestim}
	\left| \varphi\left(R_{\oDelta}\right) \right| \leq  C_9 \Delta t^2.
\end{equation}
Collecting \cref{thm:weak:onestep:proof:1stestim,thm:weak:onestep:proof:2ndtestim,thm:weak:onestep:proof:3rdtestim} reveals the statement.
\end{proof}
\begin{theorem}\label{weak-main-thm}
	Let $T>0$ and consider the free Euler Maruyama Method (\ref{kap2-def-freeEM}) with starting value  $X(0)=X_0\in\mathcal{A}^{sa}$. 	Under the assumptions of \cref{thm:weakonestep}
	 the method (\ref{kap2-def-freeEM}) is weakly convergent with order $p=1$, i.e.
	\begin{equation}\label{eq:weak-error-fem-order1}
		\sup_{0\leq t_k \leq T}|\varphi(f(X_{k}))-\varphi(f(X_{k}))|\leq C_{f,T} \Delta t
	\end{equation}
	for all functions $f\in W_2(\R)$.
\end{theorem}
\begin{proof}
The proof copies from \cite{milsteintretyakov}, Theorem 2.1.
\end{proof}
\begin{remark} In the classical setting of commutative stochastic differential equations the weak order of convergence $p$ is valid for functions $f\in C^{2(p+1)}(\R)$ (\cite{milsteintretyakov}). In the non-commutative setting for $p=1$ we require $f\in W_2(\R)$.
\end{remark}

\section{Examples}
\label{sec:Examples}
In this section we consider the numerical solution of several free differential equations taken from \cite{kargin} and \cite{freeCIR}. We compare the numerically determined spectral distribution with theoretical results and numerically verify strong and weak convergence properties of fEMM.\\
Before we start with examples it is necessary to note some details of the implementation of fEMM and the realization of the free Brownian motion.
The implementation of fEMM acts on the von Neumann Algebra $\mathcal{M}_N(\R)$ of random matrices (see the bottom row in diagram \cref{fig:freeEM-diagramm}). The implementation of fEMM starts by dividing the interval $[0,T]$ into $L=2^l$, $l\in \N$ intervals with stepsize $\Delta t=T/L$ ($T>0$). A free Brownian motion $(W_t)_{t\geq 0}$ is then realized on each time point $t_i=i\Delta t, i=0,\dots,L$. Implementation of fEMM requires generation of increments of the free Brownian motion. We generate $L$ matrices
%\begin{equation}\label{eq:gen:freebrownmotion}
$	\Delta W_i=\sqrt{\frac{\Delta t}{2N}}(A+A^T), \, i=1,\dots,L,$
%\end{equation}
where $A=(a_{ij})$ is an $N\times N$ Matrix with independent and standard normally distributed elements $a_{ij}=N(0,1)$. These matrices $\Delta W_i$ are interpreted as the increments $W(t_i)-W(t_{i-1}), \, i=1,\dots,L$
of the free Brownian motion $(W_t)_{t\geq0}$ on the interval $[t_{i-1},t_i[, \, i=1,\dots,L$. The increments $\Delta W_i$ are free from each other and have variance $\Delta t$. To determine the order of strong convergence of fEMM numerically, we first generate $M\in\N$ number of paths and then evaluate the $L_1(\varphi)$-norm of $\oX_L^N-X_T^N$ for each path at the end point $T=L\Delta t>0$. Calculating the expectation over the number $M>0$ of paths by
$
	\mathbb{E}\left(  \text{tr}\left(  \left|\oX_L^N-X_T^N\right|  \right)     \right)/N, 
$
this value is taken as an approximation to the strong error defined by \cref{def-strong-converg-fEMM} on matrix level $\mathcal{M}_N^{sa}(\R)$. To overcome the limitation that the exact $\mathcal{M}_N^{sa}(\R)$-valued solution $X_T^N$ (as $\Delta t \rightarrow 0$) is in general unknown, we choose a minimal time step $\Delta t_{min}$ and calculate $\oX_L^N$ ($T=L\Delta t_{min}$), where the free Brownian motion realized with increments $\Delta W_i^{min}$ of variance $\Delta t_{min}$. Then we take $\oX_L^N$ as an approximation to the unknown matrix-valued solution $X_T^N$. To check the strong convergence properties we choose larger time steps $\Delta t_R=R\Delta t_{min}$ with $L/R\in\N$ and $R<L$ and employ fEMM with a corresponding free Brownian motion generated by the increments
$
	\Delta W_j^R = \sum\limits_{i=jR}^{(j-1)R} \Delta W_i^{min},\, j=1,\dots,L/R.
	%= W(t_{jR}) - W(t_{(j-1)R}) 
$
Since the increments $\Delta W_i^{min}$ are free, the variance of $\Delta W_j^R$ sum up to $\Delta t_R$. The expression 
\begin{equation}\label{eq:numerical-strong-error}
	e_s(\Delta t) = \mathbb{E}\left(  \text{tr}\left(  \left|\oX_{L/R}^N-\oX_L^N\right|  \right)     \right)/N
\end{equation}
is then taken as the strong error at time point $T=L\Delta t_{min}=\Delta t_RL/R$. 
The weak error (\ref{eq:weak-error-fem-order1}) is numerically evaluated for $f=id$ by
\begin{equation}\label{eq:numerical-weak-error}
	e_w(\Delta t)=\left| \mathbb{E}\left(\varphi\left(\oX_T^N\right)\right)-\varphi\left(X_T\right)\right|.
\end{equation}
Note that on matrix level we have $\varphi\left(\oX_T^N\right)=\text{tr}(\oX_T^N)/N$. The equations considered in the following allow the exact calculation of $\varphi(X_T)$ (for $X_t\in\mathcal{A}$ that is, including $N\rightarrow \infty$).

\subsection{Free Ornstein-Uhlenbeck Equation}
We start with the free variant of the Ornstein-Uhlenbeck equation
\begin{equation}\label{OU-equations-numerics}
	dX_t=\theta X_tdt + \sigma dW_t ,\, X_0=0, t\geq 0
\end{equation}
where $\theta,\sigma\in\R$. This equation was studied analytically in \cite{kargin} by deriving and solve a partial differential equations for the Cauchy transform of $X_t$. It turns out that the solution $X_t$ is semicircle at each time point $t\geq0$. For $\theta>0$ the time dependent radius is given by 
%\begin{equation}\label{OU-semicircle-infty}
$	R(t)=\sqrt{\frac{2\sigma^2}{\theta}\left(e^{2\theta t}-1\right)}.$
%\end{equation}
For the cases $\theta\leq0$ we refer to \cite{kargin}.
\cref{fig:ou-distri} shows the empirical probability density function of the eigenvalues of $\oX_{1024}^{500}$ calculated by fEMM for $T=1$ with a time step $\Delta t=2^{-10}$ and matrix size of $N=500$. The red line in \cref{fig:ou-distri} shows the semicircle distribution for the case $N\rightarrow \infty$.
\begin{figure}[tbhp]
	\centering
	\begin{tikzpicture}[scale=0.4, declare function={h(\z)=sqrt(3.575^2-\z^2);}]
		%	\pgfplotsset{set layers}
		\begin{axis}[ grid=major] %ymin=0, title=\texttt{hist=cumulative}]
			%\addplot [hist=cumulative] table [y index=0] {random.tsv};
			\addplot +[
			hist={density,
				bins=20,
				data min=-3.7,
				data max=3.7,			
			}, 
			mark=none, yscale=20, line width=1pt, color=black, fill=blue, fill opacity=0.3
			] table [y index=0] {OU/Xt.dat};
			\addplot [smooth, domain=-3.575:3.575, red, line width=1.0pt] {h(x)};
		\end{axis}
	\end{tikzpicture}
	\caption{Distribution of the eigenvalues of the solution $\oX_L^N$ at $T=1$ of \cref{OU-equations-numerics} for $\theta=\sigma=1$, $L=1024$ and $N=500$. The exact solution $X_T$ at $T=1$ is semicircle with $R\approx 3.575$.}
	\label{fig:ou-distri}
\end{figure}
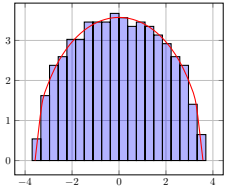
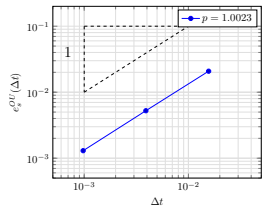
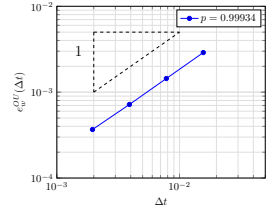
\begin{figure}[thbp]
	\centering
	\captionsetup[subfigure]{justification=centering}
	\subfloat[$\theta=1$, $\sigma=1$, $N=100$, $M=224000$]{
		\begin{tikzpicture}[scale=0.4]
			\begin{loglogaxis}[
				xlabel = {$\Delta t$},
				ylabel = {$e_s^{OU}(\Delta t)$},
				xmin=0.0005, xmax=0.05, ymin=0.0005, ymax=0.2,
				grid=both,
				minor grid style={gray!25},
				major grid style={gray!25},
				legend entries = {$p=1.0023$},
				]
				\addplot table[x index=0, y index = 1] {OU/OU-Strong.dat};
				% horizontal
				\addplot[mark=none,black, dashed] coordinates {(0.001,0.1) (0.01,0.1)};
				% vertical
				\addplot[mark=none,black, dashed] coordinates {(0.001,0.1) (0.001,0.01)};
				% slope
				\addplot[mark=none,black, dashed] coordinates {(0.001,0.01) (0.01,0.1)};
				\node[] at (axis cs: 0.0007,0.04) {\Large $1$};	
			\end{loglogaxis}
		\end{tikzpicture}
		\label{subfig:OU-strong}
	}\hspace{1cm}
	\subfloat[$\theta=2$, $\sigma=1$, $N=100$, $M=256000$]{
		\begin{tikzpicture}[scale=0.4]
			\begin{loglogaxis}[
				xlabel = {$\Delta t$},
				ylabel = {$e_w^{OU}(\Delta t)$},
				xmin=1e-3, xmax=5e-2, ymin=1e-4, ymax=1e-2,
				grid=both,
				minor grid style={gray!25},
				major grid style={gray!25},
				legend entries = {$p=0.99934$},
				]
				\addplot table[x index=0, y index = 1] {OU/OU-Weak.dat};
				% horizontal
				\addplot[mark=none,black, dashed] coordinates {(0.002,0.005) (0.01,0.005)};
				% vertical
				\addplot[mark=none,black, dashed] coordinates {(0.002,0.005) (0.002,0.001)};
				% slope
				\addplot[mark=none,black, dashed] coordinates {(0.002,0.001) (0.01,0.005)};
				\node[] at (axis cs: 0.0015,0.003) {\Large $1$};		
			\end{loglogaxis}
		\end{tikzpicture}	\label{subfig:OU-weak}
	}
	\caption{Strong and weak convergence properties of fEMM applied to the free Ornstein-Uhlenbeck  \cref{OU-equations-numerics} at $T=1$.}	\label{fig:ou-strong-weak}
\end{figure}
\cref{fig:ou-strong-weak} shows strong and weak convergence properties of fEMM applied to \cref{OU-equations-numerics}.  As discussed above we employ a minimum time step of $\Delta t_{min}=2^{-16}$ and calculate the strong error as
$
	e_s^{OU}(\Delta t) = \mathbb{E}\left(  \text{tr}\left(  \left|\oX_{L/R}^N-\oX_L^N\right|  \right)     \right)/N
$
for $L=2^{16}$ ($T=1$) and $\Delta t=R\Delta t_{min}$ with $R=6,8,10$.
\cref{subfig:OU-strong} shows strong convergence order of $p=1$.
This is not a contradiction to the expected value of $p=0.5$. If the coefficients $b,c$ of the free Brownian motion in the fSDE are constant, the fEMM shows a higher convergence order. This is an analog to the commutative case (\cite{HighamKloeden2020}).
Considering weak convergence of fEMM applied to \cref{OU-equations-numerics} is shown in \cref{subfig:OU-weak}. The weak error (\ref{eq:weak-error-fem-order1}) is numerically evaluated for $f=id$ by
$
	e_w^{OU}(\Delta t)=\left| \mathbb{E}\left(\varphi\left(\oX_L^N\right)\right)\right|
$
with $L=2^{12}$. Note that on matrix level we have $\varphi\left(\oX_L^N\right)=tr(\oX_L^N)/N$ and $\varphi\left(X_T\right)=0$, since the eigenvalue distribution of the solution $X_T$ of  \cref{OU-equations-numerics} is a centered semicircle distribution. The numerically estimated convergence order corresponds very well the theoretical value of $p=1$.
%\begin{figure}[tbhp]
%	\centering
%	\includegraphics[width=0.5\linewidth]{OU-Strong}
%	\caption{Strong convergence of Ornstein-Uhlenbeck for $\theta=1$, $\sigma=1$, $N=100$, $t=1$. $M=224000$ Pfade}
%	\label{fig:ou-strong}
%\end{figure}
%\begin{figure}[tbhp]
%	\centering
%	\includegraphics[width=0.3\linewidth]{OU-Weak}
%	\caption{Weak convergence of Ornstein-Uhlenbeck for $\theta=2$, $\sigma=1$, $N=100$, $t=1$. $M=256000$ Pfade}
%	\label{fig:ou-weak}
%\end{figure}
%\newpage
\subsection{Geometric Brownian Motion I}
Suppose that $X_t\in\mathcal{A}^{sa}$ satisfies the following equation
\begin{equation}\label{eq:geoI-numerics}
	dX_t=\theta X_tdt + X_t^{\frac{1}{2}}dW_tX_t^{\frac{1}{2}},\, X_0=I.
\end{equation}
\begin{figure}[tbhp]
	\centering
	\captionsetup[subfigure]{justification=centering}
	\subfloat[$t=100 \Delta t$]{
		\begin{tikzpicture}[scale=0.4]
			%\pgfplotsset{set layers}
			\begin{axis}[ xmax=4,ymax=1.3, grid=major] 
				\addplot +[
				hist={density,
					bins=10,
					data min=0.56,
					data max=1.94,			
				}, 
				mark=none,  line width=1pt, color=black, fill=blue, fill opacity=0.3
				] table [y index=0] {GEOI/GEOI-100-hist.dat};
				\addplot [ color=red, line width=1pt] table {GEOI/Xt_analy_100.txt};
			\end{axis}
	\end{tikzpicture}}\hspace{0.5cm}
	\subfloat[$t=300 \Delta t$]{
		\begin{tikzpicture}[scale=0.4]
			%\pgfplotsset{set layers}
			\begin{axis}[ xmax=4, ymax=1.3, grid=major] %ymin=0, title=\texttt{hist=cumulative}]
				\addplot +[
				hist={density,
					bins=20,
					data min=0.41,
					data max=3.38,			
				}, 
				mark=none,  line width=1pt, color=black, fill=blue, fill opacity=0.3
				] table [y index=0] {GEOI/GEOI-300-hist.dat};
				\addplot [ color=red, line width=1pt] table {GEOI/Xt_analy_300.txt};			
			\end{axis}
	\end{tikzpicture}}\hspace{0.5cm}
	\subfloat[$t=1024 \Delta t$]{
		\begin{tikzpicture}[scale=0.4]
			%\pgfplotsset{set layers}
			\begin{axis}[ xmax=14, ymax=0.8, grid=major] %ymin=0, title=\texttt{hist=cumulative}]
				\addplot +[
				hist={density,
					bins=30,
					data min=0.23,
					data max=13.2			
				}, 
				mark=none,  line width=1pt, color=black, fill=blue, fill opacity=0.3
				] table [y index=0] {GEOI/GEOI-1024-hist.dat};
				\addplot [ color=red, line width=1pt] table {GEOI/Xt_analy_1024.txt};			
			\end{axis}
		\end{tikzpicture}
	}
	\caption{Spectral Distribution of $\oX_t^N$  of \cref{eq:geoI-numerics} approximated by fEMM at different time points. The red line is the spectral distribution of the exact solution $X_t$ recovered from it's Cauchy transform.}
	\label{fig:geobrown1}
\end{figure}
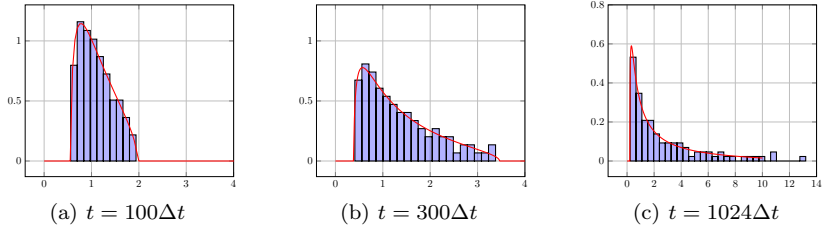
In \cite[Proposition 3.8]{kargin} it is stated that the spectral distribution of $X_t$ is supported on the interval $[I^1(t),I^2(t)]$, where
$
	I^i(t)=\frac{r_i(t)}{1+r_i(t)}e^{(\theta-1-r_i(t))}, \, 
	r_i(t)=\frac{-1\pm\sqrt{1+4/t}}{2}, \, i=1,2.
$
By applying the trace $\varphi$ to \cref{eq:geoI-numerics} it follows that $\varphi(X_t)=e^{\theta t}$. The variance of the spectral distribution is $te^{2\theta t}$ and the ratio of the standard deviation to the expectation of $X_t$ is $\sqrt{t}$ (\cite{kargin}). \cref{fig:geobrown1} shows the empirical spectral distribution of $\oX_t^N$ for $N=100$, $\Delta t=2^{-10}$ and $\theta=1$ and different time points. The red line in Figure \cref{fig:GeoI-R1R2} is the recovery of the spectral distribution for $N\rightarrow\infty$ (see \cite{kargin}). \cref{fig:GeoI-R1R2} shows that the time development of the supporting interval of the spectral distribution of $\oX_t^N$ correspond very well to the theoretical values given by $I^1(t)$ and $I^2(t)$.
\begin{figure}[tbhp]
	\centering
	\includegraphics[width=0.34\linewidth]{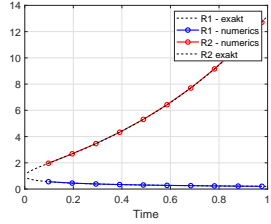}
	\caption{Comparison of boundaries of support interval $[R1=I^1(t),R2=I^2(t)]$ of the density of spectral distribution of \cref{eq:geoI-numerics} for $\theta=1$, $N=100$.}
	\label{fig:GeoI-R1R2}
\end{figure}
\begin{figure}[thbp]
	\centering
	\captionsetup[subfigure]{justification=centering}
	\subfloat[Strong convergence,\\ $N=10$, $\theta=0.1$, $M=250$]{
		\begin{tikzpicture}[scale=0.4]
			\begin{loglogaxis}[
				xlabel = {\Large $\Delta t$},
				ylabel = {\Large $e_s^{GeoI}(\Delta t)$},
				xmin=0.0001, xmax=0.1, ymin=0.005, ymax=0.2,
				grid=both,
				minor grid style={gray!25},
				major grid style={gray!25},
				legend entries = {$p=0.49989$},
				]
				\addplot table[x index=0, y index = 1] {GEOI/GEOI-Strong.txt};
				\addplot[mark=none,black, dashed] coordinates {(0.001,0.1) (0.01,0.1)};
				\addplot[mark=none,black, dashed] coordinates {(0.001,0.1) (0.001,0.031622)};
				\addplot[mark=none,black, dashed] coordinates {(0.001,0.031622) (0.01,0.1)};
				\node[] at (axis cs: 0.0007,0.06) {\Large $\frac{1}{2}$};		
			\end{loglogaxis}
		\end{tikzpicture}
		\label{fig:GeoI-Strong}
	}\hspace{1cm}
	\subfloat[Weak convergence,\\ $N=50$, $\theta=1$, $M=112000$]{
		\begin{tikzpicture}[scale=0.4]
			\begin{loglogaxis}[
				xlabel = {\Large $\Delta t$},
				ylabel = {\Large $e_w^{GeoI}(\Delta t)$},
				xmin=1e-4, xmax=5e-2, ymin=2e-4, ymax=5e-2,
				grid=both,
				minor grid style={gray!25},
				major grid style={gray!25},
				legend entries = {$p=0.9734$},
				]
				\addplot table[x index=0, y index = 1] {GEOI/GEOI-Weak.dat};
				\addplot[mark=none,black, dashed] coordinates {(1e-3,1e-2) (2e-4,1e-2)};
				\addplot[mark=none,black, dashed] coordinates {(2e-4,1e-2) (2e-4,2e-3)};
				\addplot[mark=none,black, dashed] coordinates {(2e-4,2e-3) (1e-3,1e-2)};
				\node[] at (axis cs: 5e-4,5e-3) {\Large $1$};		
			\end{loglogaxis}
		\end{tikzpicture}
		\label{fig:GeoI-Weak}
	}
	\caption{Strong and weak convergence properties of the numerical solution of \cref{eq:geoI-numerics}.}
	\label{fig:GeoI-Strong-Weak}
\end{figure}
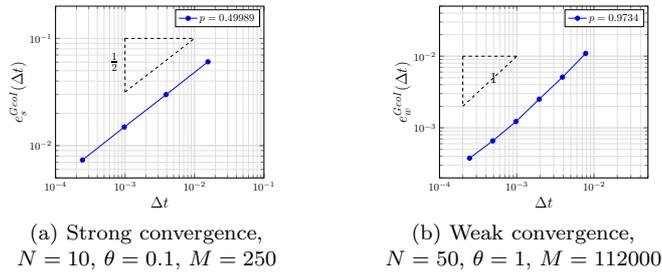

Strong and weak convergence properties of fEMM applied to \cref{eq:geoI-numerics} is shown in \cref{fig:GeoI-Strong-Weak}. The graph in \cref{fig:GeoI-Strong} shows the approximation of the strong error
$
	e_s^{GeoI}(\Delta t) = \mathbb{E}\left(  \text{tr}\left(  \left|\oX_{L/R}^N-\oX_L^N\right|  \right)     \right)/N,
$
where $\oX_L^N$ is calculated by a minimal time step of $\Delta t=2^{-16}$ and $\oX_{L/R}^N$ by time steps $R\Delta t_{min}$ with $R=6,8,10,12$. The expected value is determined over $M=250$ different paths. The slope of the straight line in \cref{fig:GeoI-Strong} shows numerically convergence order of $p=0.5$ which is in accordance to Theorem \cref{thm1}. fEMM applied to \cref{eq:geoI-numerics} shows weak convergence order $p=1$ at $T=1$ as shown in \cref{fig:GeoI-Weak}. The weak error is calculated on matrix level as
$
	e_w^{GeoI}(\Delta t)=\left| \mathbb{E}\left(\text{tr}\left(\oX_L^N\right)\right)/N-e^{\theta T}\right|
$
for $6$ different time steps $\Delta t=R/L$ ($T=1$) for $L=2^{12}$ and $R=1,2,4,8,16,32$. Again, we have good correspondence between numerical and theoretical results.

\subsection{Free CIR-Process}

Consider the equation
\begin{equation}\label{eq:CIR-SDE}
	dX_t=(a-bX_t)dt + \frac{\sigma}{2}\sqrt{X_t}dW_t+\frac{\sigma}{2}dW_t\sqrt{X_t},\, X_0=I,
\end{equation}
\begin{figure}[tbhp]
	\centering
	\captionsetup[subfigure]{justification=centering}
	\subfloat[$t=100 \Delta t$]{
		\begin{tikzpicture}[scale=0.4]
			%\pgfplotsset{set layers}
			\begin{axis}[grid=major, xmin=0, xmax=5, ymin=0, ymax=2.5] 
				\addplot +[
				hist={density,
					bins=10,
					data min=0.74,
					data max=1.36,			
				}, 
				mark=none,  line width=1pt, color=black, fill=blue,
				fill opacity=0.3
				] table [y index=0] {CIR/CIR-100-hist.dat};
			\end{axis}
	\end{tikzpicture}}\hspace{1cm}
	\subfloat[$t=4000 \Delta t$]{
		\begin{tikzpicture}[scale=0.4]
			%\pgfplotsset{set layers}
			\begin{axis}[grid=major,grid=major, xmin=0, xmax=5, ymin=0, ymax=2.5] 
				\addplot +[
				hist={density,
					bins=20,
					data min=0.42,
					data max=3.43,			
				}, 
				mark=none,  line width=1pt, color=black, fill=blue, fill opacity=0.3
				] table [y index=0] {CIR/CIR-4000-hist.dat};			
			\end{axis}
	\end{tikzpicture}}
	\caption{Time development of the empirical spectral distribution of $\oX_t^N$ of the numerical solution of the free CIR \cref{eq:CIR-SDE} with $a=2, b=1, \sigma=1, \Delta t=2^{-12}$ and $N=50$.}
	\label{fig:CIR-distri}
\end{figure}
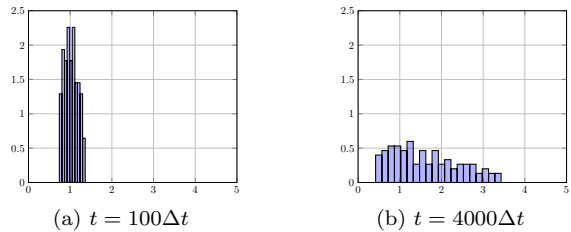
\begin{figure}[thbp]
	\centering
	\captionsetup[subfigure]{justification=centering}
	\subfloat[Strong convergence. $N=10,M=250$]{\begin{tikzpicture}[scale=0.4]
			\begin{loglogaxis}[
				xlabel = {\Large $\Delta t$},
				ylabel = {\Large $e_s^{CIR}$},
				xmin=0.0001, xmax=0.1, ymin=0.002, ymax=0.2,
				grid=both,
				minor grid style={gray!25},
				major grid style={gray!25},
				legend entries = {$p=0.502349$},
				]
				\addplot table[x index=0, y index = 1] {CIR/CIR-Strong.dat};
				\addplot[mark=none,black, dashed] coordinates {(0.001,0.1) (0.01,0.1)};
				\addplot[mark=none,black, dashed] coordinates {(0.001,0.1) (0.001,0.031622)};
				\addplot[mark=none,black, dashed] coordinates {(0.001,0.031622) (0.01,0.1)};
				\node[] at (axis cs: 0.0007,0.06) {\Large $\frac{1}{2}$};		
			\end{loglogaxis}
		\end{tikzpicture}
	}\hspace{1cm}
	\subfloat[Weak convergence, $N=50,M=112000$]{
		\begin{tikzpicture}[scale=0.4]
			\begin{loglogaxis}[
				xlabel = {\Large $\Delta t$},
				ylabel = {\Large $e_w^{CIR}$},
				xmin=1e-4, xmax=1e-1, ymin=1e-5, ymax=5e-2,
				grid=both,
				minor grid style={gray!25},
				major grid style={gray!25},
				legend entries = {$p=0.87632$},
				]
				\addplot table[x index=0, y index = 1] {CIR/CIR-Weak.dat};
				\addplot[mark=none,black, dashed] coordinates {(1e-3,1e-2) (1e-2,1e-2)};
				\addplot[mark=none,black, dashed] coordinates {(1e-3,1e-2) (1e-3,1e-3)};
				\addplot[mark=none,black, dashed] coordinates {(1e-3,1e-3) (1e-2,1e-2)};
				\node[] at (axis cs: 5e-4,5e-3) {\Large $1$};		
			\end{loglogaxis}
		\end{tikzpicture}
	}
	\caption{Strong and weak convergence properties of the free CIR \cref{eq:CIR-SDE} with parameter values $a=2, b=1, \sigma=1$}
	\label{fig:CIR-Strong-Weak}
\end{figure}
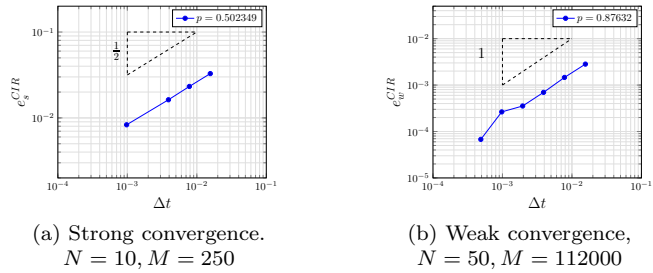
where $a,b,\sigma>0$ are such that $2a^2\geq\sigma^2$ (see \cite{freeCIR}).
The expected value can be calcuated by applying $\varphi$ to \cref{eq:CIR-SDE} which gives
$
\varphi(X_t)=\frac{e^{-bt}}{b}\left(b+a(e^{bt}-1)\right).
$
So far, no further sprectral properties are known.
The time development the spectral distribution of the numerical solution is shown in \cref{fig:CIR-distri}.  Strong and weak convergence is shown in \cref{fig:CIR-Strong-Weak}.
~\\~
\section{Conclusions}
\label{sec:conclusions}
In this paper we developed a free analog of the well known  Euler-Maruyama method. Up to the knowledge of the author the numerical treatment of fSDEs are considered for the first time. From Taylor series expansion of operator valued functions we deriveed an iterated free It\^{o} formula. This offers to motivate and define the free Euler-Maruyama method (fEMM) and proved strong and weak convergence properties. We considered the implementation of the method by approximating the elements in von Neumann algebras by self-adjoint random matrices. The method recovers well known analytical results and convergence properties where numerically verified.
%%\subsection{Jacobi Process}
\appendix
\section{Pointwise Convergence of the Cauchy transform}
\begin{lemma}\label{lemma:ChaucyPointwise}
	Let $(X_k)$ be a sequence in a von Neumann Algebra $\mathcal{A}^{sa}$. Assume that $(X_k)$ converges in the $L_1(\varphi)$-norm to $X\in\mathcal{A}^{sa}$. Then the sequence $(G_k)$ of Cauchy transforms of $X_k$ converges on $\C^+$ pointwise to the Cauchy transform $G_X$.
\end{lemma}
\begin{proof}
	Let $z\in \mathbb{C}^+$. Since $X_k\in \mathcal{A}^{sa}$ it follows that $\|(z-X_k)^{-1}\| \leq \frac{1}{Im(z)}$ by applying functional calculus to the normal element $(z-X_k)^{-1}$. The same holds for $X$. The statement follows by the estimation
	$
		\left|\varphi\left( (z-X_k)^{-1}\right)-\varphi\left( (z-X)^{-1} \right) \right| 
		\leq 	\|X_k-X\|_1 \|(z-X_k)^{-1}\| \|(z-X)^{-1}\|.
	$
\end{proof}

\bibliographystyle{siamplain}
\bibliography{references-arxiv}

\begin{thebibliography}{10}

\bibitem{ADHIKARI2022108260}
{\sc S.~Adhikari and S.~Chakraborty}, {\em Random matrix eigenvalue problems in
  structural dynamics: An iterative approach}, Mechanical Systems and Signal
  Processing, 164 (2022).

\bibitem{AnGao}
{\sc G.~An and M.~Gao}, {\em Poisson processes in free probability}, 2015,
  \url{https://doi.org/10.48550/ARXIV.1506.03130}.

\bibitem{anderson_guionnet_zeitouni_2009}
{\sc G.~W. Anderson, A.~Guionnet, and O.~Zeitouni}, {\em An Introduction to
  Random Matrices}, Cambridge Studies in Advanced Mathematics, Cambridge
  University Press, 2009, \url{https://doi.org/10.1017/CBO9780511801334}.

\bibitem{anshelevic}
{\sc M.~Anshelevich}, {\em Itô formula for free stochastic integrals}, Journal
  of Functional Analysis, 188 (2002), pp.~292--315,
  \url{https://doi.org/https://doi.org/10.1006/jfan.2001.3849}.

\bibitem{azamov_carey_dodds_sukochev_2009}
{\sc N.~A. Azamov, A.~L. Carey, P.~G. Dodds, and F.~A. Sukochev}, {\em Operator
  integrals, spectral shift, and spectral flow}, Canadian Journal of
  Mathematics, 61 (2009), p.~241–263,
  \url{https://doi.org/10.4153/CJM-2009-012-0}.

\bibitem{barnodorff10.2307/3318705}
{\sc O.~E. Barndorff-Nielsen and S.~Thorbjørnsen}, {\em Self-decomposability
  and {Lévy} processes in free probability}, Bernoulli, 8 (2002),
  pp.~323--366, \url{http://www.jstor.org/stable/3318705} (accessed
  2022-07-28).

\bibitem{BianezbMATH01003147}
{\sc P.~{Biane}}, {\em Free brownian motion, free stochastic calculus and
  random matrices}, in Free probability theory. Papers from a workshop on
  random matrices and operator algebra free products, Toronto, Canada, Mars
  1995, Providence, RI: American Mathematical Society, 1997, pp.~1--19.

\bibitem{Biane1998-2}
{\sc P.~Biane}, {\em Processes with free increments}, Mathematische Zeitschrift
  volume, 227 (1998), pp.~143--174, \url{https://doi.org/10.1007/PL00004363}.

\bibitem{Biane1998}
{\sc P.~Biane and R.~Speicher}, {\em Stochastic calculus with respect to free
  brownian motion and analysis on {Wigner} space}, Probability Theory and
  Related Fields, 112 (1998), pp.~373--409,
  \url{https://doi.org/10.1007/s004400050194}.

\bibitem{BIANESPEICHER2001581}
{\sc P.~Biane and R.~Speicher}, {\em Free diffusions, free entropy
  and free fisher information}, Annales de l'Institut Henri Poincare (B)
  Probability and Statistics, 37 (2001), pp.~581--606,
  \url{https://doi.org/10.1016/S0246-0203(00)01074-8}.

\bibitem{Bouchaud2015}
{\sc J.-P. Bouchaud and M.~Potters}, {\em Financial applications of random
  matrix theory: a short review}, The Oxford Handbook of Random Matrix Theory,
  (2015), p.~823–850.

\bibitem{ZHAOZHI}
{\sc Z.~FAN}, {\em Self-similarity of free stochastic processes}, Infinite
  Dimensional Analysis, Quantum Probability and Related Topics, 09 (2006),
  pp.~451--469, \url{https://doi.org/10.1142/S0219025706002482}.

\bibitem{GAO2006177}
{\sc M.~Gao}, {\em Free {Ornstein–Uhlenbeck} processes}, Journal of
  Mathematical Analysis and Applications, 322 (2006), pp.~177--192,
  \url{https://doi.org/https://doi.org/10.1016/j.jmaa.2005.09.013}.

\bibitem{freeCIR}
{\sc H.~Graf, H.~Port, and G.~Schlüchtermann}, {\em Free \uppercase{CIR}
  processes}, Infinite Dimensional Analysis, Quantum Probability and Related
  Topics,  (2022), \url{https://doi.org/10.1142/S0219025722500126}.

\bibitem{HighamKloeden2020}
{\sc D.~J. Higham and P.~E. Kloeden}, {\em An Introduction to the Numerical
  Simulation of Stochastic Differential Equations}, SIAM, 2021.

\bibitem{Johnstone8412585}
{\sc I.~M. Johnstone and D.~Paul}, {\em Pca in high dimensions: An
  orientation}, Proceedings of the IEEE, 106 (2018), pp.~1277--1292,
  \url{https://doi.org/10.1109/JPROC.2018.2846730}.

\bibitem{kargin}
{\sc V.~Kargin}, {\em On free stochastic differential equations}, Journal of
  Theoretical Probability,  (1998), pp.~373--409,
  \url{https://doi.org/10.1007/s10959-011-0341-z}.

\bibitem{kummererspeicher}
{\sc B.~Kummerer and R.~Speicher}, {\em {Stochastic Integration on the Cuntz
  algebra $O_\infty$}}, Journal of Funtional Analysis,  (1992), pp.~372--408.

\bibitem{Maecki2019UniversalityCF}
{\sc J.~Małecki and J.~L. P\'erez}, {\em Universality classes for general
  random matrix flows},  (2019), \url{https://arxiv.org/abs/1901.02841}.

\bibitem{milsteintretyakov}
{\sc G.~Milstein and N.~Tretyakov}, {\em Stochastic Numerics for Mathematical
  Physics}, Scientific Computation, Springer-Verlag Berlin Heidelberg, 2004,
  \url{https://doi.org/10.1007/978-3-662-10063-9}.

\bibitem{MingoSpeicher2017}
{\sc J.~A. Mingo and R.~Speicher}, {\em Free Probability and Random Matrices},
  Fields Institute Monographs, Springer, 2011,
  \url{https://doi.org/10.1007/978-1-4939-6942-5}.

\bibitem{pisier_2003}
{\sc G.~Pisier}, {\em Introduction to Operator Space Theory}, London
  Mathematical Society Lecture Note Series, Cambridge University Press, 2003,
  \url{https://doi.org/10.1017/CBO9781107360235}.

\bibitem{schwartzNonLinFANA1696}
{\sc J.~T. Schwartz}, {\em Nonlinear Functional Analysis}, Grodon and Breach
  Science Publishers, New York, London, Paris, 1969.

\bibitem{Skripka2019}
{\sc A.~Skripka and A.~Tomskova}, {\em Multiple Operator Integrals}, Springer
  International Publishing, Cham, 2019, pp.~65--112,
  \url{https://doi.org/10.1007/978-3-030-32406-3_4}.

\bibitem{Soize2017}
{\sc C.~Soize}, {\em Uncertainty Quantification}, Interdisciplinary Applied
  Mathematics, Springer, 2017, \url{https://doi.org/10.1007/978-3-319-54339-0}.

\bibitem{Speicher1990}
{\sc R.~Speicher}, {\em A new example of `independence' and `white noise'},
  Probability Theory and Related Fields, 84 (1990), pp.~141--159,
  \url{https://doi.org/10.1007/BF01197843}.

\bibitem{speicher2001free}
{\sc R.~Speicher}, {\em Free calculus}, 2001,
  \url{https://arxiv.org/abs/math/0104004}.

\bibitem{stammvoicuweber}
{\sc N.~Stammeier, D.-V. Voiculescu, and M.~Weber}, {\em Free Probability and
  Operator Algebras}, Münster Lectures in Mathematics, European Mathematical
  Society, 2016.

\bibitem{stone2018}
{\sc L.~Stone}, {\em The feasibility and stability of large complex biological
  networks: a random matrix approach}, Scientific Reports, 8 (2018),
  \url{https://doi.org/10.1038/s41598-018-26486-2}.

\bibitem{TaoBlog}
{\sc T.~Tao}, {\em Blog at wordpress.com. 254a, notes 5: Free probability,
  exercise 25},
  \url{https://terrytao.wordpress.com/2010/02/10/245a-notes-5-free-probability/}.

\bibitem{TaoIntroRMT}
{\sc T.~Tao}, {\em Topics in random matrix theory}, vol.~132 of Graduate
  Studies in Mathematics, American Mathematical Society, 2012.

\bibitem{voicudykemanica}
{\sc D.~V. Voiculescu, K.~Dykema, and A.~Nica}, {\em Free Random Variables},
  CRM monograph series, American Mathematical Society, 1992,
  \url{https://doi.org/10.1090/crmm/001}.

\bibitem{werner}
{\sc D.~Werner}, {\em Funktionalanalysis}, Springer Lehrbuch, Springer
  Spektrum, Berlin, Heidelberg, 2018,
  \url{https://doi.org/https://doi.org/10.1007/978-3-662-55407-4}.

\bibitem{XIAO2017941}
{\sc H.~Xiao, J.-X. Wang, and R.~G. Ghanem}, {\em A random matrix approach for
  quantifying model-form uncertainties in turbulence modeling}, Computer
  Methods in Applied Mechanics and Engineering, 313 (2017), pp.~941--965,
  \url{https://doi.org/10.1016/j.cma.2016.10.025}.

\bibitem{Zhang2015}
{\sc C.~Zhang and R.~C. Qiu}, {\em Massive mimo as a big data system: Random
  matrix models and testbed}, IEEE Access, 3 (2015), p.~837–851,
  \url{https://doi.org/10.1109/access.2015.2433920}.

\end{thebibliography}

%\printbibliography
\end{document}